\documentclass[12pt,letter]{article}

\usepackage{etex}
\usepackage{titlesec}
\usepackage[top=1in,bottom=1in,left=0.8in,right=0.8in]{geometry}
\usepackage[reqno,centertags]{amsmath}
\usepackage{amssymb,xspace,twoopt,mathrsfs,color, mathtools}
\usepackage[mathcal]{eucal}

\usepackage{amsmath, textcomp}
\usepackage{amsthm}
\usepackage{pstricks,pst-plot,psfrag}
\usepackage[all]{xy}
\usepackage{graphicx,subfigure}

\newcommand{\real}{{\mathbb{R}}}

\newcommand{\argmin}{\operatorname{argmin}}

\newcommand{\Tr}{\mathrm{Tr}}

\newtheorem{theorem}{Theorem}[section]
\newtheorem{definition}[theorem]{Definition}

\newtheorem{lemma}[theorem]{Lemma}

\newtheorem{corollary}[theorem]{Corollary}

\usepackage{algorithm}
\usepackage[noend]{algorithmic}
\algsetup{indent=2em}

\newcommand{\oprocendsymbol}{\hbox{$\bullet$}}
\newcommand{\oprocend}{\relax\ifmmode\else\unskip\hfill\fi\oprocendsymbol}

\begin{document}

\title{\rule{\textwidth}{0.4mm} {\textsc{Logarithmic Regret in Adaptive Control of Noisy Linear Quadratic Regulator Systems Using Hints}}\\
\rule{\textwidth}{0.4mm}}
\date{}
\maketitle
	
\vspace{-2cm}
\begin{flushright}
{\footnotesize {{\bf Mohammad Akbari}}\footnote{ Department of Mathematics and Statistics at Queen's University, \texttt{13mav1@queensu.ca}. 
}\hspace{1cm}}\\
{\footnotesize {{\bf Bahman Gharesifard}}\footnote{ Department of Electrical and Computer Engineering at University of California, Los Angeles, \texttt{gharesifard@ucla.edu}. 
}\hspace{1cm}}\\
{\footnotesize {{\bf Tamas Linder}}\footnote{ Department of Mathematics and Statistics at Queen's University, \texttt{tamas.linder@queensu.ca}. 
}\hspace{1cm}}
\end{flushright}  
\sloppy

\maketitle

%



\begin{abstract}
The problem of regret minimization for online adaptive control of linear-quadratic systems is studied. In this problem, the true system transition parameters (matrices $A$ and $B$) are unknown, and the objective is to design and analyze algorithms that generate control policies with sublinear regret. Recent studies show that when the system parameters are fully unknown, there exists a choice of these parameters such that any algorithm that only uses data from the past system trajectory at best achieves a square root of time horizon regret bound, providing a hard fundamental limit on the achievable regret in general. However, it is also known that (poly)-logarithmic regret is achievable when only matrix $A$ or only matrix $B$ is unknown. We present a result, encompassing both scenarios, showing that (poly)-logarithmic regret is achievable when both of these matrices are unknown, but a hint is periodically given to the controller. 
\end{abstract}

  
\section{Introduction}
We consider the class of linear-quadratic-regulator (LQR) systems where vector-valued control states and control actions are coupled in a linear dynamical system with quadratic cost functions. The transition dynamics of the system can be written as 
\[
x_{t+1}=A x_t+B u_t+w_t,
\]
where $x_t\in \real^n$ is the system state, $u_t\in\real^m$ is the control action at time $t\geq 1$, $\{w_t\}_{t\geq 1}$ is an i.i.d Gaussian noise sequence, and $A$ and $B$ are real matrices of appropriate dimensions. At each time $t\geq 1$, there is a quadratic cost $c_t(x_t, u_t)=x_t^\top Q_t x_t+u_t^\top R_t u_t$ associated with the system. It is well-known that the optimal controller which minimizes the expected value of the total costs over time is a linear feedback of the control state that can be derived by solving an algebraic Riccati equation~\cite{DB:2005}. This model, that represents many practical engineering problems~\cite{VVB-LC-GB-VW-VW:96,JML-LFP:12,FR-SC-FQ:13,ZKO-PB-JS:07,VVD-EAL:03}, is well-studied in optimal control theory~\cite{DB:2005,DPB-SES:96}.

A much more challenging problem, which is a standard model in adaptive control theory, is the scenario in which the dynamics of the control system (matrices $A$ and $B$) are \emph{unknown}. Most of the classical literature on this problem have focused on asymptotic results~\cite{TLL-CZW:82, HFC-LG:87, HFC-JFZ:90, MCC-PRK:98, SB-MCC:06}; however, with the advances of machine learning methods in strategic decision making settings~\cite{DS-AH:16}, robotics~\cite{TPL-JJH-AP:15}, and biology~\cite{MM-MSK-AH-SV:18}, this problem has been revisited from a learning-theoretic perspective using new tools from online learning and reinforcement learning~\cite{BR:19}. Here, by \emph{adaptive}, in accordance with the usage in e.g.~\cite{TLL-CZW:82, HFC-LG:87}, we mean the situation where the system parameters are fixed but unknown, whereas the dynamics is time-varying.


A standard notion for studying the non-asymptotic performance of algorithms for the adaptive control of LQR systems is the so-called \emph{regret}, which is defined as the difference between the accumulated cost of the controller policy generated by the algorithm and the optimal accumulated cost achievable by a controller that knows the costs and dynamics of the system. The objective here is to design algorithms that achieve sublinear regret. 
To the best of the authors' knowledge, the first result on regret analysis of adaptive control of LQR systems can be found in~\cite{YA-CS:11}. In this work, a so-called "optimism in the face of uncertainty" approach is used to design an online algorithm that achieves a regret bound of $\mathcal{O}(\sqrt{T})$. Subsequently, the dependency on dimensionality was improved by~\cite{ME-AJ-BR:12}. A new algorithm using semi-definite programming relaxation is designed by~\cite{AC-TK-YM:19}, and~\cite{HM-ST-BR:19}. The algorithm employs an $\epsilon$-greedy exploration approach to improve computational efficiency. In another approach to regret minimization for online LQR, the Thompson sampling (TS) method is used~\cite{YO-MG-RJ:17,MA-AL:17,MA-AL:18,TK-SL-KA-AA-BH:22}. In particular, \cite{TK-SL-KA-AA-BH:22} proposed an efficient TS algorithm that attains an $\mathcal{O}(\sqrt{T})$ regret bound without assuming a known stabilizing controller. 
In a related setting of online control, the linear transition dynamics is assumed to be known but the cost functions are unknown and time-varying~\cite{AC-AH-TK-NL-YM-KT:18, NA-BB-EH-SK-KS:19, NA-EH-KS:19, MA-BG-TL:20-pmlr}. Here the best known regret bound is logarithmic.

In the converse direction, \cite{MS-DF:20} proved that for any adaptive LQR problem with unknown system parameters $A$ and $B$, and for any algorithm, there is a choice of the system parameters for which the algorithm must suffer a regret at least $\Omega(\sqrt{T})$, providing a general fundamental limit. Moreover, they demonstrated that the square root regret bound with an optimal dependence on the system dimension is achievable. In the positive direction, \cite{AC-AC-TK:2020} showed that a (poly)-logarithmic regret bound is achievable if one makes the extra assumption that either $B$ is known or $A$ is known (along with a minor additional assumption). These results, however, leave open the question whether one can achieve a (poly)-logarithmic regret with milder assumptions. The main objective of this work is to investigate if (poly)-logarithmic regret can be achieved when both of these matrices are unknown, but some information about the system is given to the controller as a \emph{hint}.


\textbf{Contribution:} 
We consider adaptive linear quadratic regulator control in a scenario where the true system parameters of the transition dynamics (matrices $A$ and $B$ are \emph{unknown}, but a hint about the matrix $B$ (or $A$) is given to the controller periodically. This hint, which can be viewed as a noisy directional information pointing toward $B$ (or $A$), will help the controller to achieve logarithmic regret, even though it does not know the true system parameters $A$ and $B$. This extra directional information is in the same spirit as the notion of hint by~\cite{OD-AF-NH-PJ:17} for online optimization. Our algorithm, which uses a regularized least squared error estimate, is adopted from the work of~\cite{AC-AC-TK:2020}. However, the analysis of our algorithm with the hint is naturally more complicated. We also prove that the results of~\cite{AC-AC-TK:2020}, where one of the system parameters, $A$ or $B$, is known, can be obtained from our setting.

\textbf{Mathematical Notation:}
We denote the set of real numbers by $\mathbb{R}$. We use $\|\cdot\|$ and $\|\cdot\|_F$ to denote the 2-norm and the Frobenius norm, respectively. For a matrix $A$, we use $A^\top$ and $\Tr(A)$ to denote the transpose and trace of $A$, respectively. We let $I_n$ denote the $n\times n$ identity matrix. The Gaussian distribution with mean vector $s\in\real^d$ and covariance matrix $R\in\real^{d\times d}$ is denoted by $\mathcal{N}(s, R)$. For a matrix $A$, $\rho(A)$ denotes its spectral radius. We use the notation $A\succeq 0$ ($A\succ 0$) to indicate that $A$ is positive semi-definite (positive definite). We also use $A\succeq B$ ($A\succ B$) to indicate that $A-B$ is positive semi-definite (positive definite). The indicator function of event $\mathcal{E}$ is denoted by $\mathbf{1}\{\mathcal{E}\}$. We use $\mathrm{poly}(z)$ to denote a polynomial function of variable $z$.

\section{Background and Problem Statement}
Here, we review the problem of the adaptive control of discrete-time LQR systems. We describe the LQR problem first.
\subsection{Discrete-Time Linear Quadratic Regulator Control}

The discrete-time control of a LQR system is a classical problem in optimal control theory. This problem is modelled as follows. Let $x_t\in \real^n$ be the system state and let $u_t\in\real^m$ be the control action at time $t\geq 1$ with initial state $x_1$. The system states evolve over time according to the difference equation
\begin{equation}\label{eq102}
x_{t+1}=A_* x_{t}+B_*u_{t}+w_t,
\end{equation}
where $A_*\in\real^{n\times n}$, $B_*\in\real^{n\times m}$ are the state-state and state-action real matrices, respectively, and $\{w_t\}_{t\geq 1}$ is an i.i.d Gaussian noise sequence ($ w_t\sim\mathcal{N}(0,\sigma^2 I_n)$). At each time step $t$, there is a cost of the form $x_t^\top Q x_t+u_t^\top R u_t$, where $Q\in\real^{n\times n}$ and $R\in\real^{m\times m}$ are positive definite matrices. The objective is to find the optimal controller that minimizes the infinite horizon cost
\begin{equation}
J(\{u_t\}_{t\geq 1})=\lim_{T\rightarrow\infty}\mathbb{E}\Big[\frac{1}{T}\sum_{t=1}^T x_t^\top Q x_t+u_t^\top R u_t\Big],
\end{equation}
whenever the limit exists.

It can be proved (see for instance~\cite{DPB-SES:96}) that the optimal controller is a linear feedback of the state; i.e., $u^*_t=K_* x_t$, where $K_*\in\real^{m\times n}$ is the so-called optimal feedback gain matrix given by
\begin{equation}
K_*=-(B_*^\top P_* B_*+R)^{-1}B_*^\top P_* A_*\label{dare},
\end{equation}
where $P_*\in \real^{n\times n}$ is the solution to the algebraic Riccati equation
\begin{equation}
P_*=A_*^\top P_* A_*-A_*^\top P_* B_* (B_*^\top P_* B_* + R)^{-1}B_*^\top P_* A_* + Q,
\end{equation}
whenever the assumptions that $(A_*,B_*)$ is stabilizable and $(A_*, Q^{1/2})$ is detectable, are satisfied. Furthermore, the linear feedback gain $K_*$ is stabilizing (i.e., $(A_*+B_* K_*)$ is stable: $\rho(A_*+B_* K_*)< 1$), and the optimal infinite horizon cost can be shown to be
\begin{equation}
J_*=\min_{\{u_t\}_{t\geq 1}} J = \sigma^2 \Tr(P_*).
\end{equation}
We use the shorthand notation $K_*=\mathrm{dare}(A_*,B_*,Q,R)$ for the controller given by \eqref{dare}.

\subsection{Adaptive Control}
In adaptive control, the system parameters $(A_*,B_*)$ are not known, and the controller uses the past history of data $\{x_t, u_t\}_{t\geq 1}$ to update its policy. In this setting, although the optimal controller is potentially not achievable, we are interested in designing a controller that performs well in comparison to the optimal infinite horizon cost. It is by now customary to define a notion of {\it regret} as the difference between the cumulative cost over time when the controller uses policy $u_t=\pi_t(x_t)$ and the optimal infinite horizon cost, given by
\begin{equation}
\mathcal{R}_T(\pi)=\sum_{t=1}^T (x_t^\top Q x_t+u_t^\top R u_t)-TJ_*,
\end{equation}
where $\pi=(\pi_1, \pi_2,\ldots,\pi_T)$ and $\pi_t$ depends on the history of the data $\{x_s, u_s\}_{s=1}^{t-1}$. 

The objective here is to design an algorithm that generates a policy $\pi$ with a low regret.

\section{Main Result}
We investigate the setting where, $B_*$ and $A_*$ are not known, but a hint about one of them is provided periodically, allowing the controller to generate better estimates. This hint, which includes some noisy information about the direction towards $B_*$ or $A_*$, will help the controller to achieve logarithmic regret. We present our algorithm with hint on $B_*$ in Section~\ref{sec:Bhint}, and a similar algorithm with hint on $A_*$ in Section~\ref{sec:Ahint}. We now specify our assumptions. 

We assume that the matrices $Q$ and $R$ are known to the controller and their eigenvalues are bounded by constants $\alpha_0$, and $\alpha_1>0$, i.e.,
\begin{align*}
\alpha_0 I_n \preceq Q,R \preceq \alpha_1 I_m.
\end{align*}
We also assume that the controller knows a stabilizing feedback gain $K_0$ with finite cost such that $J(K_0)\leq \nu$ for some $\nu>0$ and hence $J_*\leq \nu$.
We also define $\phi$ as
\begin{equation*}
\phi=\max(\|A_*\|_F, \ \|B_*\|_F).
\end{equation*}
All these assumptions are standard and widely used in this setting, see~\cite{AC-AC-TK:2020}.

Next, we recall the notion of strong stability of feedback gain which is necessary in the computation in our results.

\begin{definition}
A feedback gain $K$ is $(k,\ell)$-strongly stabilizing for the system $(A,B)$ if $\|K\|\leq k$ and there exist $n\times n$ real matrices $H\succ 0$ and $L$ such that $A+BK=HLH^{-1}$ and $\|L\|\leq 1-\ell$ and $\|H\|\|H^{-1}\|\leq k$.
\end{definition}

\section{Hint for $B_*$}\label{sec:Bhint}
In this section, we present our (randomized) algorithm for the setting where a hint for $B_*$ is given to the controller. The algorithm starts with a warm-up stage where the controller uses the policy $u_t=K_0 x_t+ \eta_t$ for $1\leq t< \tau_1$, where $K_0$ is a $(k,\ell)$ strongly stabilizing feedback gain (with $k$ and $\ell$ determined by the assumption $J(K_0)\leq \nu$, see Lemma~\ref{Lem41}), $\eta_t$ is i.i.d. Gaussian perturbation ($ \eta_t\sim\mathcal{N}(0,\sigma^2 I)$), and $\tau_1$ is the length of the warm-up. The perturbation $\eta_t$ ensures with high probability the persistency of excitation of the controller~\cite{JCW-PR-IM-BLMD:05}. This notion is classically being used for the purpose of identification, see~\cite{RC:87,MG-JBM:86}. 

The data $\{x_t, u_t\}_{t=1}^{\tau_1-1}$ collected after the warm-up period is used to obtain an estimate of $(A_*,B_*)$ using regularized least squares. Let us call this estimate $(\widehat{A}_{\tau_1},\widehat{B}_{\tau_1})$. After this step, an external ``hint'', to be made precise shortly, is given to the controller that improves the estimate of $(A_*, B_*)$. This new estimate will be denoted by $(A_{\tau_1},B_{\tau_1})$. The controller uses this estimate to find a feedback gain $K_{\tau_1}$ obtained by solving the discrete algebraic Riccati equation, and implements the control input $u_t=K_{\tau_1} x_t$ for $\tau_1 \leq t <\tau_2$ with the condition that the system state remains bounded by a given constant. Otherwise, the algorithm aborts and the controller uses the policy $u_t=K_0 x_t$ until $t=T$. This process will be repeated for each time period of length $\tau_i$, and after period $\tau_i$, the controller achieves a better feedback gain $K_{\tau_i}$. We will show that the generated feedback gain results in a logarithmic regret bound. 

We now provide a description of the notion of hint and how it is used by the controller; the idea of hint used here is similar to the one in~\cite{OD-AF-NH-PJ:17}.
We assume that there is an external hint for $B_*$ given to the controller periodically. In particular, we assume that after the warm-up, the controller has the estimate $\widehat{B}_{\tau_1}$ of $B_*$. The hint is given as $\gamma_1(B_*-\widehat{B}_{\tau_1})+E_1$, where $\gamma_1 \in (0,1)$ and $E_1\in\real^{n\times m}$ are not known to the controller. The controller updates its estimate as 
\begin{equation}\label{Hhint}
B_{\tau_1}=\widehat{B}_{\tau_1}+\gamma_1(B_*-\widehat{B}_{\tau_1})+E_1.
\end{equation}
Using the new estimate $B_{\tau_1}$, the controller updates $A_{\tau_1}$ as a new estimate of $A_*$ and uses the algebraic Riccati equation to compute a new feedback gain $K_{\tau_1}$. Then the controller uses the policy $u_t=K_{\tau_1}x_t$ to control the system. We assume that at each round of estimation $\tau_i$, after making the estimate $\widehat{A}_{\tau_i}$ and $\widehat{B}_{\tau_i}$, the controller receives the hint $\gamma_i(B_*-\widehat{B}_{\tau_i})+E_i$ and revises $B_{\tau_i}=\widehat{B}_{\tau_i}+\gamma_i(B_*-\widehat{B}_{\tau_i})+E_i$. The conditions on $\gamma_i$ and $E_i$ will be provided in the main result.

\begin{algorithm}
	\caption{\textbf{Online Adaptive Control with Hint for $B_*$}}
	\label{alg}
\begin{algorithmic}[1]
  \REQUIRE a stabilizing controller $K_0$, time horizon $T$, time window parameter $r$, $\tau_1$, $k$, $x_b$, $\lambda$
  
  \STATE \textbf{Initialize} $n_T=\lfloor\log_r(T/\tau_1)\rfloor \qquad \tau_{n_{T+1}}=T+1$ 
  \STATE set $\tau_i=\tau_1 r^{2(i-1)}$  for all $i=1,\ldots,n_T$
  \FOR{each $t=1,\ldots,\tau_1-1$}
  	\STATE receive $x_t$
  	\STATE use controller $u_t=K_0 x_t+\eta_t$ 
  \ENDFOR	
  \FOR{each $i=1,2,\ldots,n_T$}
  	\STATE $(\widehat{A}_{\tau_i} \ \widehat{B}_{\tau_i})=\argmin_{(A \ B)}\sum_{t=1}^{\tau_i-1}\|x_{t+1}-Ax_t-Bu_t\|^2+\lambda\|(A \ B)\|_F^2$
  	\STATE receive hint and update $B_{\tau_i}=\widehat{B}_{\tau_i}-\gamma_i(\widehat{B}_{\tau_i}-B_*)+E_i$
  	\STATE update $A_{\tau_i}=\argmin_{A}\sum_{t=1}^{\tau_i-1}\|x_{t+1}-Ax_{t}-B_{\tau_i}u_t\|^2+\lambda\|A\|_F^2$
  	\STATE $K_{\tau_i}=\mathrm{dare}(A_{\tau_i},B_{\tau_i},Q,R)$
  \FOR{each $t=\tau_i,\ldots,\tau_{i+1}-1$}
  	\IF{$\|x_t\|^2>x_b$ or $\|K_{\tau_i}\|>k$ then} 
  		\STATE abort and play $u_t=K_0 x_t$ until $t=T$
  	\ELSE
  	\STATE play $u_t=K_{\tau_i}x_t$
  	\ENDIF
  \ENDFOR
  \ENDFOR	
\end{algorithmic}
\end{algorithm}

Our main result in this section shows that if partial information in this form about the matrix $B_*$ is provided to the controller periodically, it can achieve a logarithmic regret.

\begin{theorem}\label{mthm}
In Algorithm~\ref{alg}, let
\begin{align*}
k=\sqrt{\frac{\nu+\epsilon_0^2 C_0}{\alpha_0\sigma^2}}, \tau_1=\left\lceil\frac{240\lambda(1+\phi^2)((1+k^2)/\min\{p,1\}+1)(n+m)}{\epsilon_0^2\sigma^2}\right\rceil\\
x_b=135 n k^2 \sigma^2 \max\Big\{(1+\phi)^2k^6,4k^6\Big\}\log(4T), \lambda=(1+k)^2 x_b, p=\frac{r^2}{2+k^2},
\end{align*}
and assume $0\leq 1-\gamma_1\leq \frac{1}{r^2}$, and $\gamma_i$ satisfies $(1-\gamma_{i+1})\leq \frac{1}{r^2}(1-\gamma_i)$, and $\|E_i\|^2_F\leq \frac{1-\gamma_i}{\tau_i}$. Then for $T\geq \mathrm{poly}(\alpha_0, \alpha_1, \phi, \nu, m, n , r)$ we have \[\mathbb{E}[\mathcal{R}_T]\leq \mathrm{poly} (\alpha_0, \alpha_1, \phi, \nu, m, n , r)\log^2(T).\] In particular,
\begingroup
\allowdisplaybreaks
\begin{align}
\mathbb{E}[\mathcal{R}_T]\leq\,& \frac{135}{\log(r)}\Big((r-1)C_0 (240(1+k^2)(1+\phi^2)\big(\frac{1+k^2}{\min\{p,1\}}+1\big)(n+m)+\sigma^2)\nonumber\\
\,&+4 \alpha_1 k^6\sigma^2\Big)nk^2\max\{(1+\phi)^2k^6,4k^6\}\log^2(T)\nonumber\\
\,&+135\frac{240(1+k^2)(1+\phi^2)((1+k^2)/\min\{p,1\}+1)(n+m)+\sigma^2}{\epsilon_0^2}\nonumber\\
\,&\qquad nk^2\max\{(1+\phi)^2k^6,4k^6\}(1+\phi^2)\nu\log(4T) \nonumber\\
\,&+\frac{270}{\log(r)}\Big((r-1)C_0 (240(1+k^2)(1+\phi^2)(\frac{1+k^2}{\min\{p,1\}}+1)(n+m)+\sigma^2)\nonumber\\
\,&\qquad +4 \alpha_1 k^6\sigma^2\Big)nk^2\max\{(1+\phi)^2k^6,4k^6\}\log(T)\nonumber\\
\,&+(\nu+270\alpha_1 nk^4\sigma^2\max\{(1+\phi)^2k^6,4k^6\}\log(4T))T^{-1}\nonumber\\
\,&+1080\alpha_1 (1+8\phi^2)nk^{10}\sigma^2\max\{(1+\phi)^2k^6,4k^6\}\log(4T) T^{-2}.\nonumber
\end{align}
\endgroup
\end{theorem}

For the special case of $\gamma_i=1$ and $E_i=0$, we recover the following known result as a special case of our main theorem.
\begin{corollary}\cite[Theorem~1]{AC-AC-TK:2020}
There exists an algorithm that given matrix $B_*$ as an input, has an expected regret of $
\mathbb{E}[\mathcal{R}_T]\leq \mathrm{poly} (\alpha_0, \alpha_1, \phi, \nu, m, n , r)\log^2(T)$.
\end{corollary}

The notion of hint introduced above can be simplified to a form where a sequence of estimates $B_{\tau_i}$ is given to the controller such that $\|B_{\tau_i}-B_*\|\leq \frac{\epsilon_0\tau_1}{\tau_i}$. Specifically, for $B_{\tau_i}$ satisfying this bound, we can find $\gamma_i$ and $E_i$ that satisfy the equation $B_{\tau_i}=\widehat{B}_{\tau_i}-\gamma_i(\widehat{B}_{\tau_i}-B_*)+E_i$, where $\widehat{B}_{\tau_i}$ is an estimate of $B_*$ that can be obtained from Line $7$ of Algorithm~\ref{alg} for a sequence of $\{u_t\}_{t\geq 1}$. Using Lemma~\ref{lem210}, we have that $\|\widehat{B}_{\tau_i}-B_*\|\leq \epsilon_0$. Since $\|B_{\tau_i}-B_*\|\leq \frac{\epsilon_0\tau_1}{\tau_i}= \epsilon_0 r^{-2i}$, we have that $1-\gamma_i\leq r^{-2i}$ and we can choose $E_i$ such that $\|E_i\| \leq \frac{1-\gamma_i}{\tau_i}$. Now we can see that $\gamma_i(\widehat{B}_{\tau_i}-B_*)+E_i$ is a hint for the controller, and it achieves a logarithmic regret using Theorem~\ref{mthm}. Thus, the given sequence of estimates $B_{\tau_i}$ is sufficient for a logarithmic regret.

Note that in Algorithm~\ref{alg} we make an initial estimate $\widehat{B}_{\tau_i}$ on Line 7 using a least-square error estimate before receiving the hint. A natural question is whether this estimate is necessary to achieve a logarithmic regret. In particular, one can consider an alternative algorithm that uses an arbitrary estimate $\widehat{B}_{\tau_i}$, receives the hint on the direction towards $B_*$ (Line 8 of Algorithm~\ref{alg}), and updates the estimate $A_{\tau_i}$ using Line 9 of Algorithm~\ref{alg}. Our simulation studies suggest that this algorithm may lead to a similar regret bound. We prove this for the scalar case in Appendix B; however, the analysis in the non-scalar case appears to require a new set of tools as the current proof does not apply. We leave investigating this for the future.

The strong assumptions on the hint ($\gamma_i$ and $E_i$) make the estimate $B_{\tau_i}$ converge to $B_*$ at a speed of $1/{\tau_i}$. This assumption is used to obtain a convergence rate for the sequence $(A_{\tau_i}, B_{\tau_i})$ in Lemma~\ref{lem210}. Note that without the hint, to get better estimates of $A_*$ and $B_*$ over time, the controller needs persistency of excitation after the warm-up stage. This can be obtained by adding a random perturbation to the control. This random perturbation, which results in exploration, adds a cost that scales proportional to the covariance. The trade-off between exploration and exploitation yields a square-root regret~\cite{MS-DF:20}. To obtain a logarithmic regret, the covariance of the random perturbation needs to converge to zero as $1/T$; otherwise, the added cost by the perturbation grows larger than $\log(T)$ and imposes a larger regret; however, the data received after using this perturbation is not informative enough to improve the estimates. Hence, we need to use some extra information that allows for better estimates of $A_*$ and $B_*$, and the hint in Algorithm~\ref{alg} is sufficient for this purpose. Finding other notions of hint that allow for sufficiently good estimates for logarithmic regret is left for future work.

The proof of the theorem is organized as follows. We start by reviewing some results from the literature that play major role in proving Theorem~\ref{mthm}. Next, we state our results on how the estimated parameters $(A_{\tau_i} \ B_{\tau_i})$ are related to the history of observed data in Proposition~\ref{Prop1}. We employ Lemma~\ref{thm1} to give a bound for the error of the estimated parameters in Lemma~\ref{lem107}. We use Lemma~\ref{thm20} to obtain a lower bound for the matrix of the observed data that is stated in Lemma~\ref{lem1}. We use these results and Lemma~\ref{thm1} to find a suitable event with high probability on which the difference between $(A_{\tau_i} \ B_{\tau_i})$ and $(A_* \ B_*)$ is small, which allows us to use Lemma~\ref{lemm} to bound the regret. Then we define an event on which the system noise and controller perturbation are bounded and as a result, we have stabilizing controller and bounded states. The rest of the proof bounds the expected regret by conditioning on this event and its complement.

The following two lemmas play an important role in proving our main result. 

\begin{lemma}[\cite{HM-ST-BR:19}]\label{lemm}
There are explicit constant $C_0>0,\epsilon_0=\mathrm{poly}(\alpha_0,\alpha_1,\phi, \nu, n, m)$ such that, for any $\epsilon\in(0,\epsilon_0)$ and matrices $A\in\real^{n\times n}$ and $B\in\real^{n\times m}$ such that $\|A-A_*\|\leq \epsilon$ and $\|B-B_*\|\leq \epsilon$, the policy $K=\mathrm{dare}(A,B,Q,R)$ satisfies 
\begin{equation}
J(K)-J_*\leq C_0 \epsilon^2 \qquad and \qquad \|K-K_*\|\leq C_0 \epsilon.
\end{equation}
\end{lemma}

This lemma is used to bound the regret, when the estimated matrices $(A_{\tau_i} \ B_{\tau_i})$ are in a small neighbourhood of the true pair $(A_* \ B_*)$. The next lemma is used to get a bound on the difference of $(A_{\tau_i} \ B_{\tau_i})$ and $(A_* \ B_*)$.

\begin{lemma}[\cite{YA-CS:11}]\label{thm1}
Let $\{\mathcal{F}_t\}_{t=0}^\infty$ be a filtration and let $\{\xi_t\}_{t=1}^\infty$ be a real-valued martingale difference sequence adapted to this filtration such that $\xi_t$ is $g$-sub-Gaussian conditioned on $\mathcal{F}_{t-1}$, that is, 
\begin{equation}
\mathbb{E}[e^{\lambda \xi_t}| \mathcal{F}_{t-1}]\leq e^{\lambda^2 g^2/2}, 
\end{equation} 
for all $t\geq 1$. Further, let $\{u_t\}_{t=1}^\infty$ be an $\real^{n}$-valued stochastic process adapted to $\{\mathcal{F}_{t-1}\}_{t=1}^\infty$, let $V\in\real^{n \times n}$ be a positive-definite matrix, and for $t\geq 1$ define 
\begin{equation}
U_t=\sum_{s=1}^{t-1}\xi_s u_s, \qquad V_t=V+\sum_{s=1}^{t-1} u_s u_s^\top.
\end{equation}
Then, for any $\delta\in(0,1)$, we have that with probability at least $1-\delta$, 
\begin{equation}
U_t^\top V_t^{-1} U_t \leq 2 g^2 \log\Big(\frac{1}{\delta}\frac{\det(V_t)}{\det(V)}\Big), \qquad  t\geq 1.
\end{equation}
\end{lemma}

Throughout, we denote 
\begin{align}\label{Delta}
\Delta_t=(A_t \ B_t)-(A_* \ B_*)
\end{align}
and $z_s=(x_s^\top \ u_s^\top)^\top$. We also define $W_t\in\real^{(n+m)\times(n+m)}$ for $t\geq 2$ by
\begin{align}
W_{t}=\,&\sum_{s=1}^{t-1}z_s z_s^\top+\lambda I_{n+m} \nonumber\\ =\,&\begin{pmatrix}
\sum_{s=1}^{t-1}x_s x_s^\top \,& \sum_{s=1}^{t-1}x_s u_s^\top \\[10pt]
\sum_{s=1}^{t-1}u_s x_s^\top \,& \sum_{s=1}^{t-1}u_s u_s^\top 
\end{pmatrix}+ \lambda I_{n+m},\label{eq:wt}
\end{align}
and 
\begin{align}
V_{t}=\,&\sum_{s=1}^{t-1}x_s x_s^\top + \lambda I_n \label{eq11},\\
Y_{t}=\,&\sum_{s=1}^{t-1}u_s u_s^\top + \lambda I_m - (\sum_{s=1}^{t-1}u_s x_s^\top) V_{t}^{-1}(\sum_{s=1}^{t-1}x_s u_s^\top). \label{eq:yt}
\end{align}
Finally, we let 
\begin{align}
\widehat{W}_{\tau_i}=\begin{pmatrix}
\sum_{s=1}^{\tau_i-1}x_s x_s^\top \,& \sum_{s=1}^{\tau_i-1} x_s u_s^\top\\[10pt]
\sum_{s=1}^{\tau_i-1}u_s x_s^\top \,& \sum_{s=1}^{\tau_i-1} u_s u_s^\top + \frac{\gamma_i}{1-\gamma_i} Y_{\tau_i}
\end{pmatrix}+ \lambda I_{n+m}.\label{what}
\end{align}

The next lemma states the relation of $(A_{\tau_i} \ B_{\tau_i})$ with the history of observed data $\{x_s,u_s\}_{s=1}^{\tau_i-1}$ collected as $\widehat{W}_{\tau_i}$. 

\begin{lemma}\label{Prop1}
Let $\{x_t, u_t\}_{t=1}^T$, the sequence of states and actions of the system~\eqref{eq102}, and $(A_{\tau_i} \ B_{\tau_i})$ be generated by Algorithm~\ref{alg}. We have that
\begin{align}\label{eq106}
(A_{\tau_i} \ B_{\tau_i})=(A_* \ B_*)-\lambda (A_* \ B_*)\widehat{W}^{-1}_{\tau_i}+\big(\sum_{s=1}^{\tau_i-1}w_s (x_s^\top \ u_s^\top)\big)\widehat{W}^{-1}_{\tau_i}+\big(0 \ \frac{1}{1-\gamma_i}E_i Y_{\tau_i}\big)\widehat{W}^{-1}_{\tau_i},
\end{align}
where $\widehat{W}_{\tau_i}$ is given in~\eqref{what}.
\end{lemma}
\begin{proof}
Let $(\widehat{A}_{\tau_i},\widehat{B}_{\tau_i})$ be the estimates from Line 7 of the Algorithm~\ref{alg}. We have the following equality,
\begin{align}
(\widehat{A}_{\tau_i} \ \widehat{B}_{\tau_i})=\,&\argmin_{(A \ B)} \sum_{s=1}^{\tau_i-1}\|x_{s+1}-(A \ B)(x_s^\top \ u_s^\top)^\top\|^2+\lambda\|(A \ B)\|_F^2 \nonumber\\
=\,&\sum_{s=1}^{\tau_i-1}(x_{s+1}x_s^\top \ x_{s+1}u_s^\top)W^{-1}_{\tau_i},\label{eq101}
\end{align}
where $W_{\tau_i}$ is given in~\eqref{eq:wt}. 
Using \eqref{eq11}, \eqref{eq:yt}, and block matrix inverse formula~\cite[Theorem 2.1]{TTL-SHS:02}, we can write the inverse of $W_{\tau_i}$ as
\begin{equation}\label{Winv}
W_{\tau_i}^{-1}=\begin{pmatrix}V_{\tau_i}^{-1} + V_{\tau_i}^{-1}(\sum_{s=1}^{\tau_i-1}x_s u_s^\top) Y_{\tau_i}^{-1}(\sum_{s=1}^{\tau_i-1}u_s x_s^\top) V_{\tau_i}^{-1} \,& - V_{\tau_i}^{-1}(\sum_{s=1}^{\tau_i-1}x_s u_s^\top) Y_{\tau_i}^{-1} \\[10pt]
- Y_{\tau_i}^{-1}(\sum_{s=1}^{\tau_i-1}u_s x_s^\top) V_{\tau_i}^{-1} \,& Y_{\tau_i}^{-1}
\end{pmatrix}.
\end{equation}
Using \eqref{eq101}, \eqref{Winv}, and \eqref{Hhint}, we have that
\begin{align}
B_{\tau_i}=\,&(1-\gamma_i)\widehat{B}_{\tau_i}+\gamma_i B_*+E_i\nonumber\\
=\,&-(1-\gamma_i)(\sum_{s=1}^{\tau_i-1}x_{s+1}x_s^\top) V_{\tau_i}^{-1}(\sum_{s=1}^{\tau_i-1}x_s u_s^\top) Y_{\tau_i}^{-1}+(1-\gamma_i)(\sum_{s=1}^{\tau_i-1}x_{s+1}u_s^\top)Y_{\tau_i}^{-1}\nonumber\\
\,&+\gamma_i B_*+E_i\label{Btau}.
\end{align} 
Now for $A_{\tau_i}$, we have
\begin{align}
A_{\tau_i}=\,&\argmin_{A}\sum_{s=1}^{\tau_i-1}\|x_{s+1}-Ax_{s}-B_{\tau_i}u_s\|^2+\lambda\|A\|_F^2\nonumber\\
=\,&\Big(\sum_{s=1}^{\tau_i-1}(x_{s+1}-B_{\tau_i}u_s)x_s^\top\Big)\Big(\sum_{s=1}^{\tau_i-1}x_s x_s^\top+\lambda I_n\Big)^{-1}\nonumber\\
=\,&(\sum_{s=1}^{\tau_i-1}x_{s+1}x_s^\top)V_{\tau_i}^{-1}
-B_{\tau_i}(\sum_{s=1}^{\tau_i-1}u_s x_s^\top)V_{\tau_i}^{-1}\nonumber\\
=\,&(\sum_{s=1}^{\tau_i-1}x_{s+1}x_s^\top)V_{\tau_i}^{-1}+(1-\gamma_i) (\sum_{s=1}^{\tau_i-1}x_{s+1}x_s^\top) V_{\tau_i}^{-1}(\sum_{s=1}^{\tau_i-1}x_s u_s^\top) Y_{\tau_i}^{-1}(\sum_{s=1}^{\tau_i-1}u_s x_s^\top)V_{\tau_i}^{-1}\nonumber\\
\,&-(1-\gamma_i)(\sum_{s=1}^{\tau_i-1}x_{s+1}u_s^\top)Y_{\tau_i}^{-1}(\sum_{s=1}^{\tau_i-1}u_s x_s^\top)V_{\tau_i}^{-1}-(\gamma_i B_*+E_i)(\sum_{s=1}^{\tau_i-1}u_s x_s^\top)V_{\tau_i}^{-1},\label{Atau}
\end{align}
where the first equality follows by the Line 9 of the Algorithm~\ref{alg}, the third equality follows by~\eqref{eq11}, and the last equality follows by plugging in $B_{\tau_i}$ from~\eqref{Btau}. Putting~\eqref{Atau} and~\eqref{Btau} together, we have
\begin{equation}
(A_{\tau_i} \ \ B_{\tau_i})=\Big(\sum_{s=1}^{\tau_i-1}x_{s+1}(x_s^\top \ u_s^\top)\Big)Z_{\tau_i}+\Big(-(\gamma_i B_*+E_i)(\sum_{s=1}^{\tau_i-1}u_s x_s^\top)V_{\tau_i}^{-1} \ \  \gamma_i B_*+E_i\Big),
\end{equation}
where
\begin{equation}
Z_{\tau_i}=\begin{pmatrix}V_{\tau_i}^{-1} + (1-\gamma_i) V_{\tau_i}^{-1}(\sum_{s=1}^{\tau_i-1}x_s u_s^\top) Y_{\tau_i}^{-1}(\sum_{s=1}^{\tau_i-1}u_s x_s^\top) V_{\tau_i}^{-1} \,& - (1-\gamma_i) V_{\tau_i}^{-1}(\sum_{s=1}^{\tau_i-1}x_s u_s^\top) Y_{\tau_i}^{-1} \\[10pt]
- (1-\gamma_i) Y_{\tau_i}^{-1}(\sum_{s=1}^{\tau_i-1}u_s x_s^\top) V_{\tau_i}^{-1} \,& (1-\gamma_i) Y_{\tau_i}^{-1}
\end{pmatrix}.
\end{equation}
After some simplifications, we obtain
\begin{align*}
(A_{\tau_i} \ B_{\tau_i})=\big(\sum_{s=1}^{\tau_i-1}x_{s+1}x_s^\top \ \  \sum_{s=1}^{\tau_i-1}x_{s+1}u_s^\top+\frac{1}{1-\gamma_i}(\gamma_i B_*+E_i) Y_{\tau_i}\big)Z_{\tau_i}.
\end{align*}
Now using \eqref{eq102}, we have
\begin{align}
(A_{\tau_i} \ B_{\tau_i})=\,&\big(\sum_{s=1}^{\tau_i-1}(A_* x_s + B_* u_s + w_s)x_s^\top \ \ \sum_{s=1}^{\tau_i-1}(A_* x_s + B_* u_s + w_s)u_s^\top+\frac{1}{1-\gamma_i}(\gamma_i B_*+E_i) Y_{\tau_i}\big)Z_{\tau_i}\nonumber\\
=\,&(A_* \ B_*)\begin{pmatrix}
\sum_{s=1}^{\tau_i-1}x_s x_s^\top \,& \sum_{s=1}^{\tau_i-1} x_s u_s^\top\\[10pt]
\sum_{s=1}^{\tau_i-1}u_s x_s^\top \,& \sum_{s=1}^{\tau_i-1} u_s u_s^\top + \frac{\gamma_i}{1-\gamma_i} Y_{\tau_i}
\end{pmatrix}Z_{\tau_i}+\big(\sum_{s=1}^{\tau_i-1}w_s (x_s^\top \ u_s^\top)\big)Z_{\tau_i}\nonumber\\[10pt]
\,&+\big(0 \ \frac{1}{1-\gamma_i}E_i Y_{\tau_i}\big)Z_{\tau_i}\label{eq113}.
\end{align}
For further simplifications, let $\widehat{W}_{\tau_i}$ be given by~\eqref{what}, and $\widehat{Y}_{\tau_i}$, the Schur complement of matrix $\widehat{W}_{\tau_i}$, given by
\begin{align}
\widehat{Y}_{\tau_i}=\,&\sum_{s=1}^{\tau_i-1} u_s u_s^\top +\lambda I_m+ \frac{\gamma_i}{1-\gamma_i} Y_{\tau_i}-(\sum_{s=1}^{\tau_i-1}u_s x_s^\top) V_{\tau_i}^{-1}(\sum_{s=1}^{\tau_i-1}x_s u_s^\top)\nonumber\\
=\,&(1+\frac{\gamma_i}{1-\gamma_i})Y_{\tau_i}\nonumber\\
=\,&\frac{1}{1-\gamma_i}Y_{\tau_i}.\label{eq:yhat}
\end{align}
Now by computing the inverse of the $\widehat{W}_{\tau_i}$, we obtain
\begin{equation}
\widehat{W}_{\tau_i}^{-1}=\begin{pmatrix}V_{\tau_i}^{-1} + V_{\tau_i}^{-1}(\sum_{s=1}^{\tau_i-1}x_s u_s^\top) \widehat{Y}_{\tau_i}^{-1}(\sum_{s=1}^{\tau_i-1}u_s x_s^\top) V_{\tau_i}^{-1} \,& - V_{\tau_i}^{-1}(\sum_{s=1}^{\tau_i-1}x_s u_s^\top) \widehat{Y}_{\tau_i}^{-1} \\[10pt]
- \widehat{Y}_{\tau_i}^{-1}(\sum_{s=1}^{\tau_i-1}u_s x_s^\top) V_{\tau_i}^{-1} \,& \widehat{Y}_{\tau_i}^{-1}
\end{pmatrix},
\end{equation}
where $V_{\tau_i}$ is given by \eqref{eq11}.
Using this and $\widehat{Y}_{\tau_i}^{-1}=(1-\gamma_i){Y}_{\tau_i}^{-1}$ by~\eqref{eq:yhat}, we obtain
\begin{equation}\label{zwhat}
Z_{\tau_i}=\widehat{W}_{\tau_i}^{-1}.
\end{equation}
Plugging \eqref{zwhat} into \eqref{eq113}, we conclude~\eqref{eq106}.
\end{proof}

\begin{lemma}\label{lem107}
Let $\{x_t, u_t\}_{t=1}^T$ be the sequence of states and actions of the system~\eqref{eq102}, and let $(A_{\tau_i} , B_{\tau_i})$ be the corresponding pair generated by Algorithm~\ref{alg}. Then we have, with probability $1-\delta$,
\begin{align*}
\Tr(\Delta_{\tau_i}\widehat{W}_{\tau_i}\Delta_{\tau_i}^\top)\leq 6 n \sigma^2 \log\Big(\frac{n}{\delta}\frac{\det(W_{\tau_i})}{\det(\lambda I)}\Big)+3\lambda\|(A_\star \ B_\star)\|_F^2+\frac{3}{1-\gamma_i}\Tr(E_i Y_{\tau_i}E_i^\top),
\end{align*}
where $\Delta_{\tau_i}$ is defined by~\eqref{Delta}.
\end{lemma}
\begin{proof}
The proof strategy is similar to \cite[Lemma 6]{AC-AC-TK:2020}. Let
\[
U_{\tau_i}=\sum_{s=1}^{\tau_i-1}w_s (x_s^\top \ \ u_s^\top).
\]
Using Lemma~\ref{Prop1} and by multiplying both side of~\eqref{eq106} by $\widehat{W}_{\tau_i}$ and then multiplying each side by the transpose of corresponding side of~\eqref{eq106}, we have that
\begin{equation}
\Tr(\Delta_{\tau_{i}}\widehat{W}_{\tau_i}\Delta_{\tau_{i}}^\top)\leq 3\Tr(U_{\tau_i} \widehat{W}^{-1}_{\tau_i} U_{\tau_i}^\top)+3\lambda^2\Tr((A_{\star} \ \ B_{\star})\widehat{W}^{-1}_{\tau_i}(A_{\star} \ \ B_{\star})^\top)+\frac{3}{1-\gamma_i}\Tr(E_i Y_{\tau_i}E_i^\top).
\end{equation}
Note that 
\begin{equation}
\widehat{W}_{\tau_i}-{W}_{\tau_i}=\begin{pmatrix}
0 \,& 0\\[10pt]
0 \,& \frac{\gamma_i}{1-\gamma_i} Y_{\tau_i}
\end{pmatrix}\succeq 0,
\end{equation}
since $Y_{\tau_i}$ is the Schur complement of $W_{\tau_i}\succeq \lambda I$ and is positive definite, and $0<\gamma_i<1$. Hence, we have $\widehat{W}_{\tau_i}\succeq{W}_{\tau_i}$ and therefore $\Tr(U_{\tau_i} \widehat{W}^{-1}_{\tau_i} U_{\tau_i}^\top)\leq \Tr(U_{\tau_i} {W}^{-1}_{\tau_i} U_{\tau_i}^\top)$. Thus
\begin{align}
\Tr(\Delta_{\tau_{i}}\widehat{W}_{\tau_i}\Delta_{\tau_{i}}^\top)\leq\,& 3\Tr(U_{\tau_i} W^{-1}_{\tau_i} U_{\tau_i}^\top)+3\lambda^2\Tr((A_{\star} \ B_{\star})\widehat{W}^{-1}_{\tau_i}(A_{\star} \ B_{\star})^\top)+\frac{3}{1-\gamma_i}\Tr(E_i Y_{\tau_i}E_i^\top)\nonumber\\
\leq\,&3\Tr(U_{\tau_i} W^{-1}_{\tau_i} U_{\tau_i}^\top)+3\lambda\|(A_\star \ B_\star)\|_F^2+\frac{3}{1-\gamma_i}\Tr(E_i Y_{\tau_i}E_i^\top)\label{eq202},
\end{align}
where we have used $\widehat{W}_{\tau_i}\succeq \lambda I$ in the last inequality. In order to bound the first term, we use Theorem~\ref{thm1}. Let $U_t(i)=\sum_{s=1}^{t-1}w_s(i)(x_s^\top \ u_s^\top)$ for all $i=1,\ldots,n$. For each $i$, by Theorem~\ref{thm1}, with probability $1-\delta/n$, we have
\begin{align}\label{eq201}
U_t(i)W^{-1}_t U^\top_t(i)\leq 2\sigma^2 \log\Big(\frac{n}{\delta}\frac{\det(W_t)}{\det(\lambda I_{n+m})}\Big).
\end{align}
From this, by applying the union bound over the events for which~\eqref{eq201} is satisfied, we have, with probability $1-\delta$,
\begin{align*}
\Tr(U_t W^{-1}_t U^\top_t)=\sum_{i=1}^n U_t(i)W^{-1}_t U^\top_t(i)\leq  2n\sigma^2 \log\Big(\frac{n}{\delta}\frac{\det(W_t)}{\det(\lambda I_{n+m})}\Big),
\end{align*}
for all $i=1,\ldots,n$. The result now follows by plugging this into~\eqref{eq202}.
\end{proof}

We recall the following result.

\begin{lemma}\cite[Theorem 20]{AC-AC-TK:2020} \label{thm20} Let $\{z_t\}_{t=1}^\infty$ be a sequence of random vectors adapted to a filtration $\{\mathcal{F}_{t}\}_{t=0}^\infty$. Suppose that the $z_t$ are conditionally Gaussian given $\mathcal{F}_{t-1}$ and that $\mathbb{E}[z_t z_t^\top|\mathcal{F}_{t-1}]\succeq \sigma^2 I$ for some fixed $\sigma^2>0$. Then for $t\geq 200n \log\frac{12}{\delta}$ we have, with probability at least $1-\delta$,
\begin{equation}
\sum_{s=1}^t z_s z_s^\top \succeq \frac{t\sigma^2}{40}I.
\end{equation}
\end{lemma}

The next result will be used in proving Lemma~\ref{lem1}.
\begin{lemma}\label{lem2}
Let $K\in\real^{m\times n}$ be such that $\|K\|\leq k$ and let $p>0$. Then 
\begin{equation}
\begin{pmatrix}
I_n \,& K^\top\\
K \,& KK^\top + pI_m
\end{pmatrix}\succeq \frac{1}{(1+k^2)/p+1}I_{n+m}.
\end{equation}
\end{lemma}
\begin{proof}
Since $\|K\|\leq k$, we have $KK^\top\preceq k^2 I_m$, and hence 
\begin{align*}
KK^\top \,&\preceq k^2 I_m +\frac{p}{1+k^2+p}I_m\\
\,&= \frac{(1+k^2)(p+k^2)}{1+k^2+p}I_m.
\end{align*}
Multiplying both sides by $\frac{-p}{1+k^2}$, we obtain
\begin{align}
\frac{-p(p+k^2)}{1+k^2+p}I_m\,&\preceq \frac{-p}{1+k^2}KK^\top\nonumber\\
\,&=KK^\top-KK^\top(1-\frac{p}{p+1+k^2})^{-1}\label{eqkk},
\end{align}
now rearranging~\eqref{eqkk} gives
\begin{align*}
KK^\top+(p-\frac{p}{1+k^2+p})I_m-KK^\top(1-\frac{p}{p+1+k^2})^{-1}\succeq 0.
\end{align*}
Note that the left-hand side is the Schur complement of 
\begin{equation}
\begin{pmatrix}
(1-\frac{p}{p+1+k^2})I_n \,& K^\top\\
K \,& KK^\top + (p-\frac{p}{1+k^2+p})I_m
\end{pmatrix}.
\end{equation}
Given now the fact that $(1-\frac{p}{p+1+k^2})I_n$ is positive-definite, the claim follows.
\end{proof}

We state a useful result next.
\begin{lemma}\label{lem1}
Let $\{x_t,u_t\}_{t=1}^{\tau_1}$ be the sequence of states and actions of the system~\eqref{eq102}, and let $\tau_1\geq 200n \log\frac{12}{\delta}$. Then we have, with probability $1-\delta$,
\begin{equation}
W_{\tau_1}\succeq \frac{\tau_1 \sigma^2}{40(2+k^2)}I_{n+m}.
\end{equation} 
\end{lemma}
\begin{proof}
Let $z_t=(x_t \ u_t)$, for $t=1,\ldots,\tau_1-1$, and consider $u_t=K_0 x_t+\eta_t$, where $K_0$ is a stabilizing feedback gain, and $\eta_t\sim \mathcal{N}(0,\sigma^2I_m)$. Given that $J(K_0)\leq \nu$, by Lemma~\ref{Lem41} in Appendix A., $K_0$ is $(k,\ell)$-strongly stable, with $k=\frac{\nu}{\alpha_0\sigma^2}$ and $\ell=\frac{1}{2k^2}$. Now we can write
\begin{align*}
\mathbb{E}[z_t z_t^\top| \mathcal{F}_{t-1}]=\,&\begin{pmatrix}
\mathbb{E}[x_s x_s^\top| \mathcal{F}_{t-1}] \,& \mathbb{E}[x_s u_s^\top| \mathcal{F}_{t-1}]\\[6pt]
\mathbb{E}[u_s x_s^\top| \mathcal{F}_{t-1}] \,& \mathbb{E}[u_s u_s^\top| \mathcal{F}_{t-1}]
\end{pmatrix}\\[6pt]
=\,&\begin{pmatrix}
\mathbb{E}[x_s x_s^\top| \mathcal{F}_{t-1}] \,& \mathbb{E}[x_s x_s^\top| \mathcal{F}_{t-1}]K_0^\top\\[6pt]
K_0\mathbb{E}[x_s x_s^\top| \mathcal{F}_{t-1}] \,& K_0\mathbb{E}[x_s x_s^\top| \mathcal{F}_{t-1}]K_0^\top+\mathbb{E}[\eta_t\eta_t^\top| \mathcal{F}_{t-1}]
\end{pmatrix}\\[6pt]
=\,&\begin{pmatrix}
I_n \\
K_0 
\end{pmatrix}\mathbb{E}[x_s x_s^\top| \mathcal{F}_{t-1}] \begin{pmatrix}
I_n \,& K_0 
\end{pmatrix}+\begin{pmatrix}
0 \,& 0\\
0 \,& \mathbb{E}[\eta_t\eta_t^\top| \mathcal{F}_{t-1}]
\end{pmatrix}\\
\succeq\,& \sigma^2\begin{pmatrix}
I_n \,& K_0^\top\\
K_0 \,& K_0 K_0^\top+I_m
\end{pmatrix},
\end{align*}
where we have used the facts that $\mathbb{E}[x_s x_s^\top| \mathcal{F}_{t-1}]\succeq \mathbb{E}[w_s w_s^\top| \mathcal{F}_{t-1}]=\sigma^2 I_n$ and $\mathbb{E}[\eta_t\eta_t^\top| \mathcal{F}_{t-1}]=\sigma^2 I_m$. Now using $\|K_0\|\leq k$ and applying Lemma~\ref{lem2}, we have that
\begin{align*}
\mathbb{E}[z_t z_t^\top| \mathcal{F}_{t-1}]\succeq\,&\frac{\sigma^2}{2+k^2} I_{n+m}.
\end{align*}
Given that $W_{\tau_1}=\lambda I_{n+m} +\sum_{s=1}^{\tau_1 -1}z_s z_s^\top$, and using Lemma~\ref{thm20}, with probability $1-\delta$, we have
\begin{equation}
W_{\tau_1}\succeq \lambda I_{n+m} + \frac{(\tau_1 -1 )\sigma^2}{40(2+k^2)}I_{n+m}\succeq \frac{\tau_1 \sigma^2}{40(2+k^2)}I_{n+m},
\end{equation}
as claimed.
\end{proof}

In order to find the bound on the expected value of the regret, we consider an event with high probability on which the growth of the regret is small. For this event, the estimated parameters of the system are close to the true system parameters, and the feedback gains generated by the algorithm are strongly stable. We give the specifics of this next.

We define the following events:
\begin{subequations}
\begin{align}
\mathcal{E}_x=\,&\Big\{\sum_{t=\tau_i}^{\tau_{i+1}-1}x_s x_s^\top\succeq \frac{(\tau_{i+1}-\tau_i)\sigma^2}{40}I_n, \ \mathrm{for} \ i=1, \ldots, n_T\Big\},\label{eve:test1}\\
\mathcal{E}_W=\,&\Big\{ W_{\tau_1}\succeq \frac{\tau_1 \sigma^2}{40(2+k^2)}I_{n+m} \Big\},\label{eve:test2}\\
\mathcal{E}_w =\,& \Big\{\max_{1\leq t \leq T}\|w_t\|\leq \sigma \sqrt{15n \log 4T}\Big\},\label{eve:test3}\\
\mathcal{E}_\eta =\,& \Big\{\max_{1\leq t \leq T}\|\eta_t\|\leq \sigma \sqrt{15n \log 4T}\Big\},\label{eve:test4}\\
\mathcal{E}_\Delta=\,&\Big\{\Tr(\Delta_{\tau_i} \widehat{W}_{\tau_i} \Delta_{\tau_i}^\top)\leq 6 n \sigma^2 \log\Big(4T^3\frac{\det(W_{\tau_i})}{\det(\lambda I)}\Big)+3\lambda\|(A_\star,B_\star)\|_F^2\nonumber\\ &\qquad \qquad \qquad \qquad +\frac{3}{1-\gamma_i}\Tr(E_i Y_{\tau_i}E_i^\top), \quad \ \mathrm{for} \ i=1,\ldots, n_T\Big\}\label{eve:test5}.
\end{align}
\end{subequations}

\begin{lemma}\label{lem:245}
Let $\mathcal{E}=\mathcal{E}_x\cap \mathcal{E}_W \cap \mathcal{E}_w \cap \mathcal{E}_\eta \cap \mathcal{E}_\Delta$, and $\tau_1 \geq 600 n \log 48 T$. Then $\mathbb{P}(\mathcal{E})\geq 1-T^{-2}$.
\end{lemma}
\begin{proof}
Let $\delta=\frac{1}{4}T^{-3}$. Using Lemma~\ref{lem1}, we have that for $\tau_1 \geq 600 n \log 48 T$, $\mathbb{P}(\mathcal{E}_W)\geq 1-\frac{1}{4}T^{-3}$. Note that $\tau_{i+1}-\tau_i\geq \tau_1$, and by using Lemma~\ref{thm20} $n_T$ times and taking a union bound, since $n_{T+1}\leq T$, we obtain $\mathbb{P}(\mathcal{E}_x \cap \mathcal{E}_W)\geq 1-\frac{1}{4}T^{-2}$.

Next, by Lemma~\ref{lem13} in Appendix A.~and letting $\delta = \frac{1}{4}T^{-2}$ with probability $1-\frac{1}{4}T^{-2}$, we have that $\max_{1\leq t \leq T}\|w_t\|\leq \sigma \sqrt{15n \log 4T}$. Similarly, $\max_{1\leq t \leq T}\|\eta_t\|\leq \sigma \sqrt{15n \log 4T}$.

For the last part, we use Lemma~\ref{lem107} by selecting $\delta=\frac{1}{4}T^{-2}$. By taking a union bound, we obtain the result.

\end{proof}

\begin{lemma}\label{lem201}
Let $\{x_t, u_t\}_{t=1}^T$ be the sequence of states and actions of the system~\eqref{eq102} generated by Algorithm~\ref{alg}. Then on the event $\mathcal{E}_w\cap\mathcal{E}_\eta$, we have $\|x_t\| \leq \sigma k^3(1+\phi)\sqrt{60n \log 4T}$ and $\|u_t\|\leq \sigma k^4(2+\phi)\sqrt{60n \log 4T}$, for all $1\leq t \leq \tau_1-1$.
\end{lemma}

\begin{proof}
For $t=1,\ldots,\tau_1-1$, the controller uses $u_t=K_0x_t+\eta_t$. By plugging this into~\eqref{eq102}, and defining $\tilde{w}_t=w_t+B_*\eta_t$, we have
\begin{equation}
x_{t+1}=A_*x_t+B_*K_0x_t+B_*\eta_t+w_t=(A_*+B_*K_0)x_t+\tilde{w}_t.
\end{equation} 
By the assumption $J(K_0)\leq \nu$, and applying Lemma~\ref{Lem41} in Appendix A., we have that $K_0$ is $(k,\ell)$-strongly stable with $k=\frac{\nu}{\alpha_0\sigma^2}$ and $\ell=1/(2k^2)$. Applying now Lemma~\ref{lem104} in Appendix A.~with $x_1=0$, we conclude that 
\begin{align}\label{eq204}
	\|x_t\| \leq  2 k^3 \max_{1\leq s \leq T}\|\tilde{w}_s\|. \qquad 1\leq t \leq \tau_1-1
\end{align}
Now, we can bound $\max_{1\leq s \leq T}\|\tilde{w}_s\|$ on the event $\mathcal{E}_w\cap\mathcal{E}_\eta$:
\begin{align*}
\max_{1\leq s \leq T}\|\tilde{w}_s\|\leq \max_{1\leq s \leq T}\|w_s\|+\|B_*\|\max_{1\leq s \leq T}\|\eta_s\|\leq \sigma(1+\phi)\sqrt{15n \log 4T}.
\end{align*}
The bound for $\|x_t\|$ follows by plugging this into~\eqref{eq204}.

Using $u_t=K_0 x_t+\eta_t$, we have
\begin{align*}
\|u_t\|\leq \,&\|K_0\|\|x_t\|+\|\eta_t\|\\
\leq \,& \sigma k^4(1+\phi)\sqrt{60n \log 4T}+\max_{1\leq s \leq T}\|\eta_s\|\\
\leq \,&\sigma k^4(1+\phi)\sqrt{60n \log 4T}+\sigma \sqrt{15n \log 4T}\\
\leq \,&\sigma k^4(2+\phi)\sqrt{60n \log 4T},
\end{align*}
as claimed.
\end{proof}

\begin{lemma}\label{lem210}
Let $0\leq 1-\gamma_1\leq \frac{1}{r}$ and $(1-\gamma_{i+1})\leq \frac{1}{r}(1-\gamma_i)$ for all $i\geq 1$, and let $p=\frac{r}{2+k^2}$. Further assume $\|E_i\|^2_F\leq \frac{1-\gamma_i}{\tau_i}$ for all $i\geq 1$. On the event $\mathcal{E}=\mathcal{E}_x\cap \mathcal{E}_W \cap \mathcal{E}_w \cap\mathcal{E}_\eta\cap \mathcal{E}_\Delta$ we have that,
\begin{enumerate}
\item $K_{\tau_i}$ is $(k,\ell)$-strongly stable, for all $1\leq i \leq n_T$;
\item $\|x_t\|^2\leq x_b$, for all $1\leq t \leq T$;
\item $\|\Delta_{\tau_i}\| \leq \epsilon_0 r^{-i+1}$, for all $1\leq i \leq n_T$;
\item $\widehat{W}_{\tau_i}\succeq \frac{\sigma^2\tau_{i}}{40((1+k^2)/\min\{p,1\}+1)}I_{n+m}$ for all $1\leq i \leq n_T$,

\end{enumerate}
where 
\begin{equation}\label{x_b}
x_b=135 n k^2 \sigma^2 \max\Big\{(1+\phi)^2k^6,4k^6\Big\}\log(4T).
\end{equation}

\end{lemma}
\begin{proof}
We use an induction on $i$. For $i=1$, using $\eqref{eq102}$, for all $s\geq 1$, we have that
\begin{equation}
\mathbb{E}[x_{s}x_{s}^\top]\succeq \mathbb{E}[w_s w_s^\top] \succeq \sigma^2 I_n.
\end{equation}
For the event $\mathcal{E}_W$, we have $W_{\tau_1}\succeq \frac{\tau_1 \sigma^2}{40(2+k^2)}I_{n+m}$. Using this, we have
\begin{align}\label{eq701}
\widehat{W}_{\tau_1}=W_{\tau_1}+\begin{pmatrix}
0 \,& 0 \\
0 \,& \frac{\gamma_1}{1-\gamma_1} Y_{\tau_1}
\end{pmatrix}\succeq \frac{\tau_1 \sigma^2}{40(2+k^2)}I_{n+m},
\end{align}
which proves the statement 4 of Lemma~\ref{lem210} for $i=1$. Next we show that 
\begin{align*}
\|\Delta_{\tau_1}\| \leq \epsilon_0,
\end{align*}
proving statement 3. First note that $\|x_t\|^2\leq x_b$, for all $1\leq t < \tau_1$ by Lemma~\ref{lem201}. By defining $z_s=(x_s^\top \ u_s^\top)^\top$, and using Lemma~\ref{lem201} again, we have that
\begin{align*}
\|z_s\|\leq\,&\|x_s\|+\|u_s\|\leq k^3(1+\phi)\sqrt{60n \log 4T}+k^4(2+\phi)\sqrt{60n \log 4T}\\
\leq\,&2k^4(2+\phi)\sqrt{60n \log 4T}\leq \sqrt{\lambda},
\end{align*}
for $1\leq s \leq \tau_1$. Then by Lemma~\ref{lem105} in Appendix A., we have 
\begin{align}\label{eq401}
\log\frac{\det(W_{\tau_1})}{\det(\lambda I_{n+m})}\leq (n+m) \log T.
\end{align}
Using $\|z_s\|\leq \sqrt{\lambda}$ and since $W_{\tau_1}=\lambda I_{n+m} +\sum_{s=1}^{\tau_1-1}z_s z_s^\top$ we obtain
\[
\|W_{\tau_1}\| \leq \lambda +\sum_{s=1}^{\tau_1-1}\|z_s\|^2\leq \tau_1 \lambda
\]
and since $Y_{\tau_1}$ is the Schur complement of $W_{\tau_1}$, we have $Y_{\tau_1}\leq \tau_1 \lambda$, and therefore
\begin{align*}
\frac{3}{1-\gamma_1}\Tr(E_1 Y_{\tau_1}E_1^\top) \leq \,&\frac{3}{1-\gamma_1}\|Y_{\tau_1}\|\Tr(E_1 E_1^\top)\\
\leq \,& \frac{3\tau_1 \lambda}{1-\gamma_1}\frac{1-\gamma_1}{\tau_1}\leq 3\lambda,
\end{align*}
where we have used the assumption that $\Tr(E_1 E_1^\top)\leq \frac{1-\gamma_1}{\tau_1}$.
For the event $\mathcal{E}_\Delta$, we have 
\begin{align*}\label{eq301}
\|\Delta_{\tau_1}\|^2\leq \,& \Tr\Big(\Delta_{\tau_1}\Delta_{\tau_1}^\top\Big)\\
\leq \,&\frac{40(2+k^2)}{\tau_1 \sigma^2}\Tr\Big(\Delta_{\tau_1}\widehat{W}_{\tau_1}\Delta_{\tau_1}^\top\Big)\\
\leq \,& \frac{40(2+k^2)}{\tau_1 \sigma^2}\Big(6 n \sigma^2 \log\Big(4T^3\frac{\det(W_{\tau_1})}{\det(\lambda I)}\Big)+6\lambda\|(A_\star \ B_\star)\|_F^2\Big),
\end{align*}
where we have used~\eqref{eq701} in the second inequality and Lemma~\ref{lem107} in the last inequality. Using this and~\eqref{eq401}, we can write
\begin{align*}
\|\Delta_{\tau_1}\|^2\leq \,&\frac{2+k^2}{\tau_1}\Big(240n(n+m)\log(4T^4)+\frac{240\lambda(n+m)\phi^2}{\sigma^2}\Big)\\
\leq \,&\frac{(2+k^2)(n+m)}{\tau_1}\Big(960n\log(4T)+\frac{240\lambda\phi^2}{\sigma^2}\Big)\\
\leq \,&\frac{(2+k^2)(n+m)}{\tau_1}\frac{240\lambda(1+\phi^2) }{\sigma^2}\leq \frac{\epsilon_0^2\tau_1}{\tau_1}\leq \epsilon_0^2,
\end{align*}
where we have used the definition of $\tau_1$ in the last inequality.

We can now apply Lemma~\ref{lem42} in Appendix A.~to conclude that $K_{\tau_1}$ is $(k, l)$-strongly stable, with $k=\sqrt{\frac{\nu+C_0\epsilon_0^2}{\alpha_0\sigma^2}}$, and $l=\frac{1}{2k^2}$. This proves the statement~1 for $i=1$. 

Next we prove the statements of Lemma~\ref{lem210} for $i=2,\ldots, n_T$.
Assume that 
\begin{equation}\label{indh}
\widehat{W}_{\tau_s}\succeq \frac{\sigma^2\tau_{s}}{40((1+k^2)/\min\{p,1\}+1)}I_{n+m},
\end{equation} and $K_{\tau_s}$ is $(k,\ell)$-strongly stable for $s\in\{1,\ldots,i\}$, we prove that this is also true for $s=i+1$.

Note that $Y_{\tau_1}$ is the Schur complement of $W_{\tau_1}$ and we have that $Y_{\tau_1}\succeq \frac{\tau_1 \sigma^2}{40(2+k^2)}I_m$. Also note that $W_{\tau_{i+1}}\succeq W_{\tau_{i}}$ for all $i\geq 1$, since
\begin{align*}
W_{\tau_{i+1}}=\,& W_{\tau_{i}}+\begin{pmatrix}
\sum_{s=\tau_i}^{\tau_{i+1}-1}x_{s}x_{s}^\top \,& \sum_{s=\tau_i}^{\tau_{i+1}-1}x_{s}u_{s}^\top\\[10pt]
\sum_{s=\tau_i}^{\tau_{i+1}-1}u_{s}x_{s}^\top \,&	\sum_{s=\tau_i}^{\tau_{i+1}-1}u_{s}u_{s}^\top 
\end{pmatrix}\\
=\,& W_{\tau_{i}}+\sum_{s=\tau_i}^{\tau_{i+1}-1}\begin{pmatrix}
x_{s} \\
u_s 
\end{pmatrix}\begin{pmatrix}
x_s^\top \,& u_s^\top
\end{pmatrix}\\
\succeq \,& W_{\tau_{i}},
\end{align*}
and hence by Lemma~\ref{lem43} in Appendix A., we have that $Y_{\tau_{i+1}}\succeq Y_{\tau_{i}}$.

Now for $\widehat{W}_{\tau_{i+1}}$, we have that
\begin{align*}
\widehat{W}_{\tau_{i+1}}=\,& \begin{pmatrix}
\sum_{s=1}^{\tau_{i+1}-1}x_s x_s^\top + \lambda I_n \,& \sum_{s=1}^{\tau_{i+1}-1} x_s u_s^\top\\[10pt]
\sum_{s=1}^{\tau_{i+1}-1}u_s x_s^\top \,& \sum_{s=1}^{\tau_{i+1}-1} u_s u_s^\top + \lambda I_m + \frac{\gamma_{i+1}}{1-\gamma_{i+1}} Y_{\tau_{i+1}}
\end{pmatrix}\\[10pt]
\succeq \,& \begin{pmatrix}
\sum_{s=1}^{\tau_{i+1}-1}x_s x_s^\top + \lambda I_n \,& \sum_{s=1}^{\tau_{i+1}-1} x_s u_s^\top\\[10pt]
\sum_{s=1}^{\tau_{i+1}-1}u_s x_s^\top \,& \sum_{s=1}^{\tau_{i+1}-1} u_s u_s^\top + \lambda I_m + \frac{\gamma_{i+1}}{1-\gamma_{i+1}} Y_{\tau_i}
\end{pmatrix}\\[10pt]
= \,& \begin{pmatrix}
\sum_{s=1}^{\tau_{i}-1}x_s x_s^\top + \lambda I_n \,& \sum_{s=1}^{\tau_{i}-1} x_s u_s^\top\\[10pt]
\sum_{s=1}^{\tau_{i}-1}u_s x_s^\top \,& \sum_{s=1}^{\tau_{i}-1} u_s u_s^\top + \lambda I_m + \frac{\gamma_{i}}{1-\gamma_{i}} Y_{\tau_i}
\end{pmatrix}\\
\,&+\begin{pmatrix}
\sum_{s=\tau_i}^{\tau_{i+1}-1}x_s x_s^\top \,& \sum_{s=\tau_i}^{\tau_{i+1}-1} x_s u_s^\top\\[10pt]
\sum_{s=\tau_i}^{\tau_{i+1}-1}u_s x_s^\top \,& \sum_{s=\tau_i}^{\tau_{i+1}-1} u_s u_s^\top + \frac{\gamma_{i+1}-\gamma_i}{(1-\gamma_{i+1})(1-\gamma_i)} Y_{\tau_i}
\end{pmatrix}\\[10pt]
\succeq \,&\widehat{W}_{\tau_{i}} + \begin{pmatrix}
\sum_{s=\tau_i}^{\tau_{i+1}-1}x_s x_s^\top \,& \sum_{s=\tau_i}^{\tau_{i+1}-1} x_s u_s^\top\\[10pt]
\sum_{s=\tau_i}^{\tau_{i+1}-1}u_s x_s^\top \,& \sum_{s=\tau_i}^{\tau_{i+1}-1} u_s u_s^\top + \frac{\gamma_{i+1}-\gamma_i}{(1-\gamma_{i+1})(1-\gamma_i)} Y_{\tau_1}
\end{pmatrix}\\[10pt]
=\,&\widehat{W}_{\tau_{i}}+\begin{pmatrix}
I_n \\ K_{\tau_i}
\end{pmatrix}
\sum_{s=\tau_i}^{\tau_{i+1}-1}x_{s}x_{s}^\top \begin{pmatrix}
I_n \,& K_{\tau_i}^\top
\end{pmatrix}+\begin{pmatrix}
0 \,& 0\\[10pt]
0 \,&	\frac{\gamma_{i+1}-\gamma_i}{(1-\gamma_{i+1})(1-\gamma_i)} Y_{\tau_1}
\end{pmatrix}
\end{align*}
where we have used the fact that $Y_{\tau_{i+1}}\succeq Y_{\tau_i}$ for the first inequality and $Y_{\tau_i}\succeq Y_{\tau_1}$ and $u_s=K_{\tau_{i}} x_s$ for the last inequality and the last equality, respectively. Using Lemma~\ref{thm20} and the fact that $\mathbb{E}[x_s x_s^\top]\succeq \sigma^2 I_n$, for the event $\mathcal{E}_x$, we have $\sum_{s=\tau_{i}}^{\tau_{i+1}-1}x_{s}x_{s}^\top\succeq \frac{(\tau_{i+1}-\tau_i)\sigma^2}{40}I_n$, and hence we have that
\begin{align*}
\widehat{W}_{\tau_{i+1}}\succeq \,& \widehat{W}_{\tau_{i}}+\frac{(\tau_{i+1}-\tau_{i})\sigma^2}{40}\begin{pmatrix}
I_n \\ K_{\tau_i}
\end{pmatrix}\begin{pmatrix}
I_n \,& K_{\tau_i}^\top
\end{pmatrix}+\begin{pmatrix}
0 \,& 0\\[10pt]
0 \,&	\frac{\gamma_{i+1}-\gamma_i}{(1-\gamma_{i+1})(1-\gamma_i)} \frac{\tau_1 \sigma^2}{40(2+k^2)}I
\end{pmatrix}\\[10pt]
= \,& \widehat{W}_{\tau_{i}}+ \frac{\sigma^2(\tau_{i+1}-\tau_i)}{40}\begin{pmatrix}
I_n \,& K_{\tau_i}^\top\\[10pt]
K_{\tau_i} \,& K_{\tau_i}K_{\tau_i}^\top + \frac{\gamma_{i+1}-\gamma_i}{(1-\gamma_{i+1})(1-\gamma_i)}\frac{\tau_1}{(\tau_{i+1}-\tau_i)}\frac{1}{2+k^2}I_m
\end{pmatrix}.
\end{align*}
Now given that $\tau_i=r^{2(i-1)}\tau_1$, we have $\frac{\tau_1}{(\tau_{i+1}-\tau_i)}=\frac{1}{(r^2-1)r^{2(i-1)}}$. Now let $\gamma_i$ be such that $(1-\gamma_{i})\leq \frac{1}{r^{2i}}$ and $1-\gamma_{i+1}\leq \frac{1}{r^2}(1-\gamma_i)$. Then we have 
\begin{align*}
\frac{\gamma_{i+1}-\gamma_i}{(1-\gamma_{i+1})(1-\gamma_i)}=\,&\frac{1}{1-\gamma_{i+1}}-\frac{1}{1-\gamma_i}\\
\geq\,& \frac{r^2-1}{1-\gamma_i}\geq (r^2-1)r^{2i}.
\end{align*}
Now by defining $p=\frac{r^2}{(2+k^2)}$ we have that \[
p\leq \frac{\gamma_{i+1}-\gamma_i}{(1-\gamma_{i+1})(1-\gamma_i)}\frac{\tau_1}{(\tau_{i+1}-\tau_i)}\frac{1}{2+k^2},\]
and also
\begin{align*}
\widehat{W}_{\tau_{i+1}}\succeq \widehat{W}_{\tau_{i}}+ \frac{\sigma^2(\tau_{i+1}-\tau_i)}{40}\begin{pmatrix}
I_n \,& K_{\tau_i}^\top\\[10pt]
K_{\tau_i} \,& K_{\tau_i}K_{\tau_i}^\top + pI_m
\end{pmatrix}.
\end{align*}
By Lemma~\ref{lem2} and using the induction hypothesis~\eqref{indh}, we have that
\begin{align*}
\widehat{W}_{\tau_{i+1}}\succeq \,&\frac{\sigma^2\tau_i}{40((1+k^2)/\min\{p,1\}+1)}I_{n+m}+ \frac{\sigma^2(\tau_{i+1}-\tau_i)}{40((1+k^2)/p+1)}I_{n+m}\\
\succeq \,&\frac{\sigma^2\tau_{i+1}}{40((1+k^2)/\min\{p,1\}+1)}I_{n+m}.
\end{align*}

Next we show that $\|x_t\|^2\leq x_b$, $\|u_t\|^2\leq k^2 x_b$, and $\|z_t\|^2\leq \lambda$ for all $\tau_i\leq t < \tau_{i+1}$, and $\|\Delta_{\tau_{i+1}}\|\leq \epsilon_0 r^{i-1}$, where $x_b$ is given by~\eqref{x_b} and $\lambda=(1+k)^2 x_b$.
First, by induction hypothesis the algorithm does not abort and the controller uses the $(k,\ell)$-strongly stable feedback gain $K_{\tau_i}$ for $\tau_i\leq t<\tau_{i+1}$. We can hence use Lemma~\ref{lem39} in Appendix A.~to obtain 
\begin{align*}
\|x_t\|\leq 3k\max\Big\{\frac{1}{2}\|x_{\tau_1}\|, \frac{k}{\ell}\max_{1\leq s\leq T}\|w_s\|\Big\}, \quad \mathrm{for \ all \ } \tau_1\leq t< \tau_{i+1}.
\end{align*}
Letting $\ell=1/(2k^2)$, and using the bound for $\|x_{\tau_1}\|$ from Lemma~\ref{lem201}, we have 
\begin{align*}
\|x_t\|\leq k\sigma \max\big\{k^3(1+\phi), 2k^3\big\}\sqrt{135n \log 4T}\leq \sqrt{x_b}.
\end{align*}
Now for $t=\tau_i,\ldots,\tau_{i+1}-1 $, we have $u_t=K_{\tau_i}x_t$ and hence $\|u_t\|=k\sqrt{x_b}$. Therefore, for $t=\tau_i,\ldots,\tau_{i+1}-1 $, we have
\begin{align*}
\|z_t\|\leq \|x_t\|+\|u_t\|\leq (1+k)\sqrt{x_b}\leq \sqrt{\lambda}.
\end{align*}
Using Lemma~\ref{lem105}, we obtain 
\begin{align*}
\log\frac{\det(W_{\tau_{i+1}})}{\det(\lambda I_{n+m})}\leq (m+n)\log T.
\end{align*}
Given that $\|z_t\|^2\leq \lambda$ and using a similar argument used for $Y_{\tau_1}$, we have that $Y_{\tau_i}\leq \tau_i \lambda$, and hence
\begin{align*}
\frac{3}{1-\gamma_i}\Tr(E_i Y_{\tau_i}E_i^\top) \leq \,&\frac{3}{1-\gamma_i}\|Y_{\tau_i}\|\Tr(E_i E_i^\top)\\
\leq \,& \frac{3\tau_i \lambda}{1-\gamma_i}\frac{1-\gamma_i}{\tau_i}\leq 3\lambda,
\end{align*}
where we have used the assumption that $\Tr(E_i E_i^\top)\leq \frac{1-\gamma_i}{\tau_i}$.
Using a similar argument for $\Delta_{\tau_1}$, we obtain
\begin{align*}
\|\Delta_{\tau_{i+1}}\|^2\leq \,&\frac{((1+k^2)/\min\{p,1\}+1)}{\tau_{i+1}}\Big(160n(n+m)\log(4T^4)+\frac{80\lambda(n+m)\phi^2}{\sigma^2}\Big)\\
\leq \,&\frac{((1+k^2)/\min\{p,1\}+1)(n+m)}{\tau_{i+1}}\Big(640n\log(4T)+\frac{80\lambda\phi^2}{\sigma^2}\Big)\\
\leq \,&\frac{((1+k^2)/\min\{p,1\}+1)(n+m)}{\tau_{i+1}}\frac{80\lambda(1+\phi)^2 }{\sigma^2}\leq \frac{\epsilon_0^2\tau_1}{\tau_{i+1}}\leq \epsilon_0^2 r^{-2i}.
\end{align*}
Finally, by applying Lemma~\ref{lem42} in Appendix A., we conclude that $K_{\tau_{i+1}}$ is $(k,\ell)$-strongly stable.
\end{proof}

\begin{proof}[of Theorem~\ref{mthm}]
Let $\mathcal{E}=\mathcal{E}_x\cap \mathcal{E}_W \cap \mathcal{E}_w \cap \mathcal{E}_\eta \cap \mathcal{E}_\Delta$ be the event defined in~\eqref{eve:test1}. The regret can be written as 
\begin{align*}
\mathbb{E}[\mathcal{R}_T]={J}_1+ {J}_2 + {J}_3-TJ_*,
\end{align*}
where 
\begin{align*}
J_1=\,&\mathbb{E}\Big[\textbf{1}\{\mathcal{E}\}\sum_{i=1}^{n_T}\sum_{t=\tau_i}^{\tau_{i+1}-1}x_t^\top Q x_t+u_t^\top R u_t\Big], \\
J_2=\,&\mathbb{E}\Big[\textbf{1}\{\mathcal{E}^c\}\sum_{t=\tau_1}^{T}x_t^\top Q x_t+u_t^\top R u_t\Big], \\
J_3=\,&\mathbb{E}\Big[\sum_{t=0}^{\tau_{1}-1}x_t^\top Q x_t+u_t^\top R u_t\Big].
\end{align*}
In the following lemmas, we will bound each term.
\begin{lemma}\label{lemj1}
$J_1\leq TJ_*+n_T((r-1)C_0 \epsilon_0^2 \tau_1+4 \alpha_1 k^6 x_b)$,
where $C_0$ and $\epsilon_0$ are positive constants given in Lemma~\ref{lemm}, and $x_b$ is given in~\eqref{x_b}.
\end{lemma}
\begin{proof}
For each $i=1,\ldots, n_T$, we define the events $\mathcal{E}_{\tau_i}=\big\{\|\Delta_{\tau_i}\|\leq \epsilon_0 r^{-i+1}\big\}$ and $\mathcal{S}_{\tau_i}=\big\{\|x_{\tau_i}\|^2\leq x_b\big\}$. We have $\mathcal{E}\subset \mathcal{E}_{\tau_i}\cap \mathcal{S}_{\tau_i}$, and
\begin{align*}
\textbf{1}\{\mathcal{E}\}\sum_{t=\tau_i}^{\tau_{i+1}-1}x_t^\top Q x_t+u_t^\top R u_t\leq \textbf{1}\{\mathcal{E}_{\tau_i}\cap \mathcal{S}_{\tau_i}\}\sum_{t=\tau_i}^{\tau_{i+1}-1}x_t^\top (Q+K_{\tau_i}^\top R K_{\tau_i}) x_t
\end{align*}
Note that $\mathcal{E}_{\tau_i}$, $\mathcal{S}_{\tau_i}$ and $K_{\tau_i}$ are completely determined by $x_{\tau_i}$, $A_{\tau_i}$, and $B_{\tau_i}$. By computing a total expectation, we have that
\begin{align*}
\mathbb{E}\Big[\textbf{1}\{\mathcal{E}\}\sum_{t=\tau_i}^{\tau_{i+1}-1}x_t^\top Q x_t+u_t^\top R u_t\Big]\leq \mathbb{E}\Big[\textbf{1}\{\mathcal{E}_{\tau_i}\cap \mathcal{S}_{\tau_i}\} \mathbb{E}\Big[\sum_{t=\tau_i}^{\tau_{i+1}-1}x_t^\top (Q+K_{\tau_i}^\top R K_{\tau_i}) x_t\Big|x_{\tau_i},A_{\tau_i},B_{\tau_i}\Big]\Big]
\end{align*}
Using Lemma~\ref{lem42} in Appendix A., for the event $\mathcal{E}_{\tau_i}$, we conclude that $K_{\tau_i}$ is $(k,\ell)$-strongly stable. Therefore, by Lemma~\ref{lem40} in Appendix A., we have that
\begin{align*}
\mathbb{E}\Big[\textbf{1}\{\mathcal{E}\}\sum_{t=\tau_i}^{\tau_{i+1}-1}x_t^\top Q x_t+u_t^\top R u_t\Big]\leq \,& (\tau_{i+1}-\tau_i)\mathbb{E}[\textbf{1}\{\mathcal{E}_{\tau_i}\}J(K_{\tau_i})]+\frac{2\alpha_1 k^4}{\ell}\mathbb{E}[\textbf{1}\{\mathcal{S}_{\tau_i}\}\|x_{\tau_i}\|^2]\\
\leq \,&(\tau_{i+1}-\tau_i)\mathbb{E}[\textbf{1}\{\mathcal{E}_{\tau_i}\}J(K_{\tau_i})]+4 \alpha_1 k^6 x_b.
\end{align*}
Now by Lemma~\ref{lemm}, we obtain 
\begin{align*}
(\tau_{i+1}-\tau_i)\mathbb{E}[\textbf{1}\{\mathcal{E}_{\tau_i}\}J(K_{\tau_i})]\leq \,& (\tau_{i+1}-\tau_i)(J_*+C_0 \epsilon_0^2 r^{-i+1})\\
\leq \,& (\tau_{i+1}-\tau_i)J_*+(r-1)C_0 \epsilon_0^2 \tau_1
\end{align*}
Summing over $i$, and since $\sum_{i=1}^{n_T}\tau_{i+1}-\tau_i\leq T$, we conclude that
\begin{align*}
J_1\leq TJ_*+n_T((r-1)C_0 \epsilon_0^2 \tau_1+4 \alpha_1 k^6 x_b).
\end{align*}
\end{proof}

The next result can be proved using the same proof as in Lemma 9 of \cite{AC-AC-TK:2020}, and hence we omit the proof.
\begin{lemma}\label{lemj2}
$J_2\leq (J(K_0)+2\alpha_1 k^2 x_b)T^{-1}+8\alpha_1 k^8(1+8\phi^2)x_b T^{-2}$.
\end{lemma}

\begin{lemma}\label{lemj3}
$J_3\leq \tau_1(1+\phi^2) J(K_0)$.
\end{lemma}
\begin{proof}
The total cost $J_3$ is the expected cost of warm-up part of Algorithm~1. Note that for this part, i.e., for $t=0,\ldots, \tau_1$, the controller uses $u_t=K_0 x_t+\eta_t$, and by plugging this into~\eqref{eq102}, we obtain
\[x_{t+1}=(A_*+B_*K_0)x_t+(B_*\eta_t+w_t).\]
This is equivalent to the system $(A_*,B_*)$ with noise covariance $\sigma^2(I+B_*B_*^\top)$, where the controller uses $u_t=K_0 x_t$. Let $J(K_0,W)$ be the infinite horizon cost for system~\eqref{eq102} with noise covariance $W$, where the controller uses $u_t=K_0 x_t$. Now we have that
\begin{equation}
J(K_0,\sigma^2(I+B_*B_*^\top))=\Tr(\sigma^2(I+B_*B_*^\top)P)\leq (1+\phi^2)\Tr(\sigma^2 P)=(1+\phi^2) J(K_0,\sigma^2).
\end{equation}
By applying Lemma~\ref{lem40} in Appendix A., we conclude that
\begin{align*}
J_3=\,&\mathbb{E}\Big[\sum_{t=0}^{\tau_{1}-1}x_t^\top Q x_t+u_t^\top R u_t\Big]\\
\leq\,&\tau_1 J(K_0,\sigma^2(I+B_*B_*^\top))+\frac{2\alpha_1 k^4}{\ell}\|x_1\|^2\\
\leq\,& \tau_1(1+\phi^2) J(K_0),
\end{align*}
where we have used the assumption that $x_1=0$.
\end{proof}
Now we apply Lemmas~\ref{lemj1}, \ref{lemj2}, and \ref{lemj3}, to conclude that
\begin{align}
\mathbb{E}[\mathcal{R}_T]=~\,&{J}_1+ {J}_2 + {J}_3-TJ_*\nonumber\\
\leq ~\,& n_T((r-1)C_0 \epsilon_0^2 \tau_1+4 \alpha_1 k^6 x_b)\nonumber\\
\,&+(J(K_0)+2\alpha_1 k^2 x_b)T^{-1}+8\alpha_1 k^8(1+8\phi^2)x_b T^{-2}\nonumber\\
\,&+\tau_1(1+\phi^2)\nu.\label{eq901}
\end{align}
Substituting the values of $\tau_1$, $x_b$, and $n_T$ in~\eqref{eq901}, we have that
\begingroup
\allowdisplaybreaks
\begin{align}
\mathbb{E}[\mathcal{R}_T]\leq\,& 135\frac{\log(T)}{\log(r)}\Big((r-1)C_0 (240(1+k^2)(1+\phi^2)(\frac{1+k^2}{\min\{p,1\}}+1)(n+m)+\sigma^2)\nonumber\\
\,&+4 \alpha_1 k^6\sigma^2\Big)nk^2\max\{(1+\phi)^2k^6,4k^6\}\log(4T)\nonumber\\
\,&+(J(K_0)+270\alpha_1 nk^4\sigma^2\max\{(1+\phi)^2k^6,4k^6\}\log(4T))T^{-1}\nonumber\\
\,&+1080\alpha_1 (1+8\phi^2)nk^{10}\sigma^2\max\{(1+\phi)^2k^6,4k^6\}\log(4T) T^{-2}\nonumber\\
\,&+135\frac{240(1+k^2)(1+\phi^2)((1+k^2)/\min\{p,1\}+1)(n+m)+\sigma^2}{\epsilon_0^2}\nonumber\\
\,&nk^2\max\{(1+\phi)^2k^6,4k^6\}\log(4T)(1+\phi^2)\nu.\nonumber
\end{align}
\endgroup
By rearranging the terms we get 
\begingroup
\allowdisplaybreaks
\begin{align}
\mathbb{E}[\mathcal{R}_T]\leq\,& \frac{135}{\log(r)}\Big((r-1)C_0 (240(1+k^2)(1+\phi^2)(\frac{1+k^2}{\min\{p,1\}}+1)(n+m)+\sigma^2)\nonumber\\
\,&+4 \alpha_1 k^6\sigma^2\Big)nk^2\max\{(1+\phi)^2k^6,4k^6\}\log^2(T)\nonumber\\
\,&+135\frac{240(1+k^2)(1+\phi^2)((1+k^2)/\min\{p,1\}+1)(n+m)+\sigma^2}{\epsilon_0^2}\nonumber\\
\,&~~nk^2\max\{(1+\phi)^2k^6,4k^6\}(1+\phi^2)\nu\log(4T) \nonumber\\
\,&+\frac{135}{\log(r)}\Big((r-1)C_0 (240(1+k^2)(1+\phi^2)(\frac{1+k^2}{\min\{p,1\}}+1)(n+m)+\sigma^2)\nonumber\\
\,&+4 \alpha_1 k^6\sigma^2\Big)nk^2\max\{(1+\phi)^2k^6,4k^6\}\log(4)\log(T)\nonumber\\
\,&+(\nu+270\alpha_1 nk^4\sigma^2\max\{(1+\phi)^2k^6,4k^6\}\log(4T))T^{-1}\nonumber\\
\,&+1080\alpha_1 (1+8\phi^2)nk^{10}\sigma^2\max\{(1+\phi)^2k^6,4k^6\}\log(4T) T^{-2}.\label{eq:regs1}
\end{align}
\endgroup
Overall, noting that the dominant term is $\log^2(T)$ in the first two lines of~\eqref{eq:regs1}, we conclude that \[\mathbb{E}[\mathcal{R}_T]\leq \mathrm{poly} (\alpha_0, \alpha_1, \phi, \nu, m, n , r)\log^2(T).\]

\end{proof}

\section{Hint for $A_*$}\label{sec:Ahint}
In this section, we present an alternative algorithm for the case where a hint about $A_*$ is given to the controller and $B_*$ is unknown. In this scenario, we need the extra assumption that $K_*K_*^\top\succ 0$. The algorithm is similar to Algorithm~\ref{alg}, with the extra warm-up period to ensure that the feedback gain $K_{\tau_i}$ achieved by the estimate $A_{\tau_i}$ and $B_{\tau_i}$ has the property that $K_{\tau_i}K_{\tau_i}^\top \succeq \mu I$ where $\mu>0$ is given in the algorithm. Similar to the hint for $B_*$ in Algorithm~\ref{alg}, the hint for $A_*$ is a noisy information towards $A_*$ that helps the controller get better estimates of system parameters. To make it more precise, after the warm-up period, the controller gets an estimate $(\widehat{A}_{\tau_i} \ \widehat{B}_{\tau_i})$ using a regularized least squares method. Then the controller receives the hint and updates a new estimate $A_{\tau_i}$ using
\[A_{\tau_i}=\widehat{A}_{\tau_i}-\gamma_i(\widehat{A}_{\tau_i}-A_*)+E_i,\]
where $0<\gamma_i<1$ and $E_i\in\real^{n \times n}$ is some error. The controller uses $A_{\tau_i}$ and the least squares method to update $B_{\tau_i}$. With the new estimates $A_{\tau_i}$ and $B_{\tau_i}$, the controller updates the feedback gain $K_{\tau_i}$. The pseudo-code of the algorithm is given in Algorithm~\ref{alg2}.

\begin{algorithm}
	\caption{\textbf{Online Adaptive Control with Hint for $A_{*}$}}
	\label{alg2}
\begin{algorithmic}[1]
  \REQUIRE a stabilizing controller $K_0$, time horizon $T$, time window parameter  $\tau_1$, $k$, $x_b$, $\lambda$, $\mu_1$
  \STATE \textbf{Initialize} $n_T=\lfloor\log_4(T/\tau_1)\rfloor \qquad \tau_{n_{T+1}}=T+1$ 
  \STATE set $\tau_i=\tau_1 4^{i-1}$ and $\mu_i=\mu_1 2^{i-1}$  for all $i=1,\ldots,n_T$
  \FOR{each $t=1,\ldots,\tau_1-1$}
  	\STATE receive $x_t$
  	\STATE use controller $u_t=K_0 x_t+\eta_t$ 
  \ENDFOR	
  \FOR{each $i=1,2,\ldots,n_T$}
  	\STATE $(\widehat{A}_{\tau_i} \ \widehat{B}_{\tau_i})=\argmin_{(A \ B)}\sum_{t=1}^{\tau_i-1}\|x_{t+1}-Ax_t-Bu_t\|^2+\lambda\|(A \ B)\|_F^2$
  	\STATE receive hint and update $A_{\tau_i}=\widehat{A}_{\tau_i}-\gamma_i(\widehat{A}_{\tau_i}-A_*)+E_i$
  	\STATE update $B_{\tau_i}=\argmin_{B}\sum_{t=1}^{\tau_i-1}\|x_{t+1}-A_{\tau_i}x_{t}-Bu_t\|^2+\lambda\|B\|_F^2$
  	\STATE $K_{\tau_i}=\mathrm{dare}(A_{\tau_i},B_{\tau_i},Q,R)$
  	\IF{$K_{\tau_i}K_{\tau_i}^\top\succeq 3\mu_i/2 I$ then}
  		\STATE set $n_s=i$ and break
  		\FOR{each $t=\tau_i,\ldots,\tau_{i+1}-1$}
  			\STATE play $u_t=K_0 x_t + \eta_t$
  		\ENDFOR
  	\ENDIF
  \ENDFOR
  \FOR{each $i=n_s,\ldots,n_T$}
  	\STATE $(\widehat{A}_{\tau_i} \ \widehat{B}_{\tau_i})=\argmin_{(A \ B)}\sum_{t=1}^{\tau_i-1}\|x_{t+1}-Ax_t-Bu_t\|^2+\lambda\|(A \ B)\|_F^2$
  	\STATE receive hint and update $A_{\tau_i}=\widehat{A}_{\tau_i}-\gamma_i(\widehat{A}_{\tau_i}-A_*)+E_i$
  	\STATE update $B_{\tau_i}=\argmin_{B}\sum_{t=1}^{\tau_i-1}\|x_{t+1}-A_{\tau_i}x_{t}-Bu_t\|^2+\lambda\|B\|_F^2$
  	\STATE $K_{\tau_i}=\mathrm{dare}(A_{\tau_i},B_{\tau_i},Q,R)$
  \FOR{each $t=\tau_i,\ldots,\tau_{i+1}-1$}
  	\IF{$\|x_t\|^2>x_b$ or $\|K_{\tau_i}\|>k$ then} 
  		\STATE abort and play $u_t=K_0 x_t$ until $t=T$
  	\ELSE
  	\STATE play $u_t=K_{\tau_i}x_t$
  	\ENDIF
  \ENDFOR
  \ENDFOR	
\end{algorithmic}
\end{algorithm}

Similar to Theorem~\ref{mthm}, in this case we can achieve a poly-logarithmic regret as the next theorem shows.
\begin{theorem}\label{thm2}
Let Algorithm~\ref{alg2} be run with parameters
\begin{align*}
k=\sqrt{\frac{\nu+\epsilon_0^2 C_0}{\alpha_0\sigma^2}}, \tau_1=\left\lceil\frac{240\lambda(1+\phi^2)\max\{(2+k^2),\frac{1}{2\mu_*p}(\mu_*+2p+2)\}(n+m)}{\epsilon_0^2\sigma^2}\right\rceil\\
x_b=135 n k^2 \sigma^2 \max\Big\{(1+\phi)^2k^6,4k^6\Big\}\log(4T), \lambda=(1+k)^2 x_b, p=\frac{4}{2+k^2}, \mu_1=4kC_0 \epsilon_0
\end{align*}
and assume $0\leq 1-\gamma_1\leq \frac{1}{4}$, $\gamma_i$ satisfies $(1-\gamma_{i+1})\leq \frac{1}{4}(1-\gamma_i)$, and $\|E_i\|^2_F\leq \frac{1-\gamma_i}{\tau_i}$. Then for $T\geq \mathrm{poly}(\alpha_0, \alpha_1, \phi, \nu, m, n , r)$ we have \[\mathbb{E}[\mathcal{R}_T]\leq \mathrm{poly} (\alpha_0, \alpha_1, \phi, \nu, m, n , r)\log^2(T).\] In particular,

\begin{align}
\mathbb{E}[\mathcal{R}_T]\leq\,& 135\Big(3C_0 (240(1+k^2)(1+\phi^2)\max\{(2+k^2),\frac{1}{2\mu_*p}(\mu_*+2p+2)\}(n+m)+\sigma^2)\nonumber\\
\,&+4 \alpha_1 k^6\sigma^2+\alpha_1 m \sigma^2 k^{8}\Big)nk^2\max\{(1+\phi)^2k^6,4k^6\}\log^2(T)\nonumber\\
\,&+135\frac{240(1+k^2)(1+\phi^2)\max\{(2+k^2),\frac{1}{2\mu_*p}(\mu_*+2p+2)\}(n+m)+\sigma^2}{\epsilon_0^2}\nonumber\\
\,&~~nk^2\max\{(1+\phi)^2k^6,4k^6\}(1+\phi^2)\nu\log(4T) \nonumber\\
\,&+{135}\Big(3C_0 (240(1+k^2)(1+\phi^2)\max\{(2+k^2),\frac{1}{2\mu_*p}(\mu_*+2p+2)\}(n+m)+\sigma^2)\nonumber\\
\,&+4 \alpha_1 k^6\sigma^2\Big)nk^2\max\{(1+\phi)^2k^6,4k^6\}\log(4)\log(T)\nonumber\\
\,&+(\nu+270\alpha_1 nk^4\sigma^2\max\{(1+\phi)^2k^6,4k^6\}\log(4T))T^{-1}\nonumber\\
\,&+1080\alpha_1 (1+8\phi^2)nk^{10}\sigma^2\max\{(1+\phi)^2k^6,4k^6\}\log(4T) T^{-2}.\nonumber
\end{align}
\end{theorem}

Similar to Lemma~\ref{Prop1}, we have the following result.
\begin{lemma}
Let $\{x_t, u_t\}_{t=1}^T$ be the sequence of states and actions generated using the Algorithm~\ref{alg2}. If $(A_{\tau_i} \ B_{\tau_i})$ are generated by Algorithm~\ref{alg2}, then we have
\begin{align}\label{eq:843}
(A_{\tau_i} \ B_{\tau_i})=(A_* \ B_*)-\lambda(A_* \ B_*)Z_{\tau_i}+\big(\sum_{s=1}^{\tau_i-1}w_s(x_s^\top \ u_s^\top)\big)Z_{\tau_i}+(\frac{E_i}{1-\gamma_i}Y_{\tau_i} \ 0)Z_{\tau_i},
\end{align}
where 
\begin{align*}
Z_{\tau_i}=\begin{pmatrix} \sum_{s=1}^{\tau_i-1}x_s x_s^\top + \lambda I + \frac{\gamma_i}{1-\gamma_i}Y_{\tau_i} & \sum_{s=1}^{\tau_i-1}x_s u_s^\top \\[10pt]
 \sum_{s=1}^{\tau_i-1}u_s x_s^\top & \sum_{s=1}^{\tau_i-1}u_s u_s^\top + \lambda I
\end{pmatrix}^{-1}
\end{align*}
and 
\begin{align*}
V_{\tau_i}=&\sum_{s=1}^{\tau_i-1}u_s u_s^\top + \lambda I_m, \\
Y_{\tau_i}=&\sum_{s=1}^{\tau_i-1}x_s x_s^\top +\lambda I_n -(\sum_{s=1}^{\tau_i-1}x_s u_s^\top)V_{\tau_i}^{-1}(\sum_{i=1}^{\tau_i-1}u_s x_s^\top).
\end{align*}
\end{lemma}

\begin{proof}
From Line 7 of Algorithm~\ref{alg2}, we have
\begin{align}
(\widehat{A}_{\tau_i} \ \widehat{B}_{\tau_i})=&\argmin_{(A \ B)}\sum_{t=1}^{\tau_i-1}\|x_{t+1}-Ax_t-Bu_t\|^2+\lambda\|(A \ B)\|_F^2\nonumber\\
=&\sum_{s=1}^{\tau_i-1}(x_{s+1}x_s^\top \ x_{s+1}u_s^\top)W_{\tau_i}^{-1},\label{eq:101}
\end{align}
where 
\begin{align*}
W_{\tau_i}^{-1}=&\begin{pmatrix} \sum_{s=1}^{\tau_i-1}x_s x_s^\top + \lambda I_n & \sum_{s=1}^{\tau_i-1}x_s u_s^\top \\
\sum_{s=1}^{\tau_i-1}u_s x_s^\top & \sum_{s=1}^{\tau_i-1}u_s u_s^\top +\lambda I_m
\end{pmatrix}^{-1}\\[10pt]
=&\begin{pmatrix}Y_{\tau_i}^{-1} & - Y_{\tau_i}^{-1}(\sum_{s=1}^{\tau_i-1}x_s u_s^\top) V_{\tau_i}^{-1} \\[10pt]
 - V_{\tau_i}^{-1}(\sum_{s=1}^{\tau_i-1}u_s x_s^\top) Y_{\tau_i}^{-1} & V_{\tau_i}^{-1} + V_{\tau_i}^{-1}(\sum_{s=1}^{\tau_i-1}u_s x_s^\top) Y_{\tau_i}^{-1}(\sum_{s=1}^{\tau_i-1}x_s u_s^\top) V_{\tau_i}^{-1} 
\end{pmatrix},
\end{align*}
and where 
\begin{align*}
V_{\tau_i}=&\sum_{s=1}^{\tau_i-1}u_s u_s^\top + \lambda I_m, \\
Y_{\tau_i}=&\sum_{s=1}^{\tau_i-1}x_s x_s^\top +\lambda I_n -(\sum_{s=1}^{\tau_i-1}x_s u_s^\top)V_{\tau_i}^{-1}(\sum_{i=1}^{\tau_i-1}u_s x_s^\top).
\end{align*}
Using \eqref{eq:101}, we have
\begin{align}
\widehat{A}_{\tau_i}=(\sum_{s=1}^{\tau_i-1}x_{s+1}x_s^\top)Y_{\tau_i}^{-1}-(\sum_{s=1}^{\tau_i-1}x_{s+1}u_s^\top)V_{\tau_i}^{-1}(\sum_{s=1}^{\tau_i-1}u_s x_s^\top)Y_{\tau_i}^{-1}\label{eq:102}.
\end{align}
From Line 8 of Algorithm~\ref{alg2} and \eqref{eq:102}, we have
\begin{align}
A_{\tau_i}=&\widehat{A}_{\tau_i}-\gamma_i(\widehat{A}_{\tau_i}-A_*)+E_i\nonumber\\
=&(1-\gamma_i)\widehat{A}_{\tau_i}+\gamma_i A_*+E_i\nonumber\\
=&(1-\gamma_i)(\sum_{s=1}^{\tau_i-1}x_{s+1}x_s^\top)Y_{\tau_i}^{-1}-(1-\gamma_i)(\sum_{s=1}^{\tau_i-1}x_{s+1}u_s^\top)V_{\tau_i}^{-1}(\sum_{s=1}^{\tau_i-1}u_s x_s^\top)Y_{\tau_i}^{-1}+\gamma_i A_*+E_i.\label{eq:103}
\end{align}
From Line 9 of Algorithm~\ref{alg2}, we have
\begin{align}
B_{\tau_i}=&\argmin_{B}\sum_{s=1}^{\tau_i-1}\|x_{s+1}-A_{\tau_i}x_{s}-Bu_s\|^2+\lambda\|B\|_F^2\nonumber\\
=&\big(\sum_{t=1}^{\tau_i-1}(x_{s+1}-A_{\tau_i}x_{s})u_s^\top\big)(\sum_{s=1}^{\tau_i-1}u_s u_s^\top+\lambda I)^{-1}\nonumber\\
=&(\sum_{t=1}^{\tau_i-1}x_{s+1}u_s^\top)(\sum_{s=1}^{\tau_i-1}u_s u_s^\top+\lambda I)^{-1}-A_{\tau_i}(\sum_{s=1}^{\tau_i-1}x_s u_s^\top)(\sum_{s=1}^{\tau_i-1} u_s u_s^\top+ \lambda I)^{-1}\nonumber\\
=&(\sum_{t=1}^{\tau_i-1}x_{s+1}u_s^\top)V_{\tau_i}^{-1}-(1-\gamma_i)(\sum_{s=1}^{\tau_i-1}x_{s+1}x_s^\top)Y_{\tau_i}^{-1}(\sum_{s=1}^{\tau_i-1}x_s u_s^\top)V_{\tau_i}^{-1}\nonumber\\
&+(1-\gamma_i)(\sum_{s=1}^{\tau_i-1}x_{s+1}u_s^\top)V_{\tau_i}^{-1}(\sum_{s=1}^{\tau_i-1}u_s x_s^\top)Y_{\tau_i}^{-1}(\sum_{s=1}^{\tau_i-1}x_s u_s^\top)V_{\tau_i}^{-1}-(\gamma_i A_*+E_i)(\sum_{s=1}^{\tau_i-1}x_s u_s^\top)V_{\tau_i}^{-1},\label{eq:104}
\end{align}
where we have used \eqref{eq:103} in the last equality. Now combining \eqref{eq:103} and \eqref{eq:104}, we obtain
\begin{align*}
(A_{\tau_i} \ B_{\tau_i})=(\sum_{s=1}^{\tau_i-1}x_{s+1}x_s^\top  \  \sum_{s=1}^{\tau_i-1}x_{s+1}u_s^\top)Z_{\tau_i}+(\frac{\gamma_i A_*+E_i}{1-\gamma_i}Y_{\tau_i} \ 0)Z_{\tau_i},
\end{align*}
where
\begin{align*}
Z_{\tau_i}=&\begin{pmatrix} (1-\gamma_i)Y_{\tau_i}^{-1} & - (1-\gamma_i)Y_{\tau_i}^{-1}(\sum_{s=1}^{\tau_i-1}x_s u_s^\top) V_{\tau_i}^{-1} \\[10pt]
 - (1-\gamma_i)V_{\tau_i}^{-1}(\sum_{s=1}^{\tau_i-1}u_s x_s^\top) Y_{\tau_i}^{-1} & V_{\tau_i}^{-1} + (1-\gamma_i)V_{\tau_i}^{-1}(\sum_{s=1}^{\tau_i-1}u_s x_s^\top) Y_{\tau_i}^{-1}(\sum_{s=1}^{\tau_i-1}x_s u_s^\top) V_{\tau_i}^{-1}
\end{pmatrix}\\[10pt]
=&\begin{pmatrix} \sum_{s=1}^{\tau_i-1}x_s x_s^\top + \lambda I + \frac{\gamma_i}{1-\gamma_i}Y_{\tau_i} & \sum_{s=1}^{\tau_i-1}x_s u_s^\top \\[10pt]
 \sum_{s=1}^{\tau_i-1}u_s x_s^\top & \sum_{s=1}^{\tau_i-1}u_s u_s^\top + \lambda I
\end{pmatrix}^{-1}.
\end{align*}
Now using~\eqref{eq102}, we have
\begin{align*}
(A_{\tau_i} \ B_{\tau_i})=&(\sum_{s=1}^{\tau_i-1}(A_*x_{s}+B_*u_s+w_s)x_s^\top  \  \sum_{s=1}^{\tau_i-1}(A_*x_{s}+B_*u_s+w_s)u_s^\top)Z_{\tau_i}+(\frac{\gamma_i A_*+E_i}{1-\gamma_i}Y_{\tau_i} \ 0)Z_{\tau_i}\\
=&(A_* \ B_*)\begin{pmatrix} \sum_{s=1}^{\tau_i-1} x_s x_s^\top + \frac{\gamma_i}{1-\gamma_i}Y_{\tau_i} & \sum_{s=1}^{\tau_i-1}x_s u_s^\top\\ \sum_{s=1}^{\tau_i-1}u_s x_s^\top & \sum_{s=1}^{\tau_i-1}u_s u_s^\top\end{pmatrix} Z_{\tau_i}+\big(\sum_{s=1}^{\tau_i-1}w_s(x_s^\top \ u_s^\top)\big)Z_{\tau_i}\\
&+(\frac{E_i}{1-\gamma_i}Y_{\tau_i} \ 0)Z_{\tau_i}\\
=&(A_* \ B_*)-\lambda(A_* \ B_*)Z_{\tau_i}+\big(\sum_{s=1}^{\tau_i-1}w_s(x_s^\top \ u_s^\top)\big)Z_{\tau_i}+(\frac{E_i}{1-\gamma_i}Y_{\tau_i} \ 0)Z_{\tau_i}.
\end{align*}
\end{proof}

As an immediate result, the following lemma can be proved similar to Lemma~\ref{lem107}.
\begin{lemma}\label{lem:502}
Let $\{x_t, u_t\}_{t=1}^T$ be the sequence of states and actions, and let $(A_{\tau_i} , B_{\tau_i})$ be the corresponding pair generated by Algorithm~\ref{alg2}. Then we have, with probability $1-\delta$,
\begin{align*}
\Tr(\Delta_{\tau_i}Z_{\tau_i}\Delta_{\tau_i}^\top)\leq 6 n \sigma^2 \log\Big(\frac{n}{\delta}\frac{\det(W_{\tau_i})}{\det(\lambda I)}\Big)+3\lambda\|(A_\star \ B_\star)\|_F^2+\frac{3}{1-\gamma_i}\Tr(E_i Y_{\tau_i}E_i^\top),
\end{align*}
where 
\begin{equation}
\Delta_{\tau_i}=(A_{\tau_i}-A_* \ \  B_{\tau_i}-B_*).
\end{equation}
\end{lemma}

Another key lemma in proving our result is presented next.
\begin{lemma}
Let $K\in\real^{m\times n}$ be such that $KK^\top\succeq \mu I_n$ and let $p>0$. Then 
\begin{equation}
\begin{pmatrix}
I_n+ pI_n \,& K^\top\\
K \,& KK^\top
\end{pmatrix}\succeq \frac{\mu p}{\mu+p+1}I_{n+m}.
\end{equation}
\end{lemma}
\begin{proof}
Since $(1+p-\frac{\mu p}{1+\mu+p})I \succ 0$, the matrix 
\begin{align*}
\begin{pmatrix}
I_n+ pI_n - \frac{\mu p}{\mu+p+1}I_n \,& K^\top\\
K \,& KK^\top-\frac{\mu p}{\mu+p+1}I_m
\end{pmatrix}
\end{align*}
is positive semi-definite, if and only if its Schur complement is positive semi-definite, i.e.,
\begin{align}
KK^\top-\frac{\mu p}{\mu+p+1}I_m - K \frac{1}{1+p-\frac{\mu p}{\mu+p+1}} I_n K^\top \succeq 0\label{eq:523}.
\end{align}
Using the assumption $KK^\top \succeq \mu I_m$, we have
\begin{align*}
KK^\top \succeq &(\mu-\frac{\mu^2 p^2}{p(1+p)(\mu+ p+1)})I_m \\
= &\frac{\mu p}{\mu +p+1}\frac{(1+p)(1+p+\mu)-\mu p}{p(1+p)}I_m.
\end{align*}
Hence, we have
\begin{align*}
(\frac{p(1+p)}{(1+p)(1+p+\mu)-\mu p})KK^\top-\frac{\mu p}{\mu+p+1}I_m \succeq 0,
\end{align*}
and 
\begin{align*}
(1-\frac{\mu+p+1}{(1+p)(1+p+\mu)-\mu p})KK^\top-\frac{\mu p}{\mu+p+1}I_m \succeq 0,
\end{align*}
which proves \eqref{eq:523}.
\end{proof}

\begin{lemma}\label{lem:510}
Let $\{x_t,u_t\}_{t=1}^{\tau_{n_s}}$ be the sequence of states and actions of the system, and let $\tau_{1}\geq 200n \log\frac{12}{\delta}$. Then for $1\leq i\leq n_s$ we have, with probability $1-\delta$,
\begin{equation}
W_{\tau_{i}}\succeq \frac{\tau_i \sigma^2}{40(2+k^2)}I_{n+m}.
\end{equation} 
\end{lemma}
The proof is similar to Lemma~\ref{lem1}.

Next, we define the events on which the growth of regret is small. 
\begin{align*}
\mathcal{E}_x=\,&\Big\{\sum_{t=\tau_i}^{\tau_{i+1}-1}x_s x_s^\top\succeq \frac{(\tau_{i+1}-\tau_i)\sigma^2}{40}I_n, \ \mathrm{for} \ 1\leq i\leq n_T\Big\},\\
\mathcal{E}_W=\,&\Big\{ W_{\tau_i}\succeq \frac{\tau_i \sigma^2}{40(2+k^2)}I_{n+m} \ \mathrm{for} \ 1\leq i\leq n_s \Big\},\\
\mathcal{E}_w =\,& \Big\{\max_{1\leq t \leq T}\|w_t\|\leq \sigma \sqrt{15n \log 4T}\Big\},\\
\mathcal{E}_\eta =\,& \Big\{\max_{1\leq t \leq T}\|\eta_t\|\leq \sigma \sqrt{15n \log 4T}\Big\},\\
\mathcal{E}_\Delta=\,&\Big\{\Tr(\Delta_{\tau_i} \widehat{W}_{\tau_i} \Delta_{\tau_i}^\top)\leq 6 n \sigma^2 \log\Big(4T^3\frac{\det(W_{\tau_i})}{\det(\lambda I)}\Big)+3\lambda\|(A_\star,B_\star)\|_F^2\nonumber\\ &\qquad \qquad \qquad \qquad +\frac{3}{1-\gamma_i}\Tr(E_i Y_{\tau_i}E_i^\top), \quad \ \mathrm{for} \ i=1,\ldots, n_T\Big\}.
\end{align*}

\begin{lemma}
Let $\mathcal{E}=\mathcal{E}_x\cap \mathcal{E}_W \cap \mathcal{E}_w \cap \mathcal{E}_\eta \cap \mathcal{E}_\Delta$, and $\tau_1 \geq 600 n \log 48 T$. Then $\mathbb{P}(\mathcal{E})\geq 1-T^{-2}$.
\end{lemma}
\begin{proof}
The proof is similar to the proof of Lemma~\ref{lem:245}.
\end{proof}
\begin{lemma}\label{lem:505}
On the event $\mathcal{E}=\mathcal{E}_x\cap \mathcal{E}_W \cap \mathcal{E}_w \cap \mathcal{E}_\eta \cap \mathcal{E}_\Delta$, we have that
\begin{enumerate}
\item $\|x_t\|\leq k^3\sigma(1+\phi)\sqrt{60n \log 4T}$ for $1\leq t \leq \tau_{n_s}-1$;
\item $\|u_t\|\leq k^4(2+\phi)\sqrt{60n \log 4T}$ for $1\leq t \leq \tau_{n_s}-1$;
\item $\|\Delta_{\tau_i}\|\leq \epsilon_0 2^{-i+1}$ for $1\leq i\leq n_s$.
\end{enumerate}
\end{lemma}

\begin{proof}
The proof for part 1 and 2 are similar to Lemma~13. Here, we prove part 3. For $1\leq i\leq n_s$, we have
\begin{align*}\label{eq:301}
\|\Delta_{\tau_i}\|^2\leq \,& \Tr\Big(\Delta_{\tau_i}\Delta_{\tau_i}^\top\Big)\\
\leq \,&\frac{40(2+k^2)}{\tau_i \sigma^2}\Tr\Big(\Delta_{\tau_i}\widehat{W}_{\tau_i}\Delta_{\tau_i}^\top\Big)\\
\leq \,& \frac{40(2+k^2)}{\tau_i \sigma^2}\Big(6 n \sigma^2 \log\Big(4T^3\frac{\det(W_{\tau_i})}{\det(\lambda I)}\Big)+6\lambda\|(A_\star \ B_\star)\|_F^2\Big),
\end{align*}
where we have used $\widehat{W}_{\tau_i}\succeq W_{\tau_i} \succeq \frac{\tau_i \sigma^2}{40(2+k^2)}I_{n+m}$ from Lemma~\ref{lem:510} in the second inequality and Lemma~\ref{lem:502} and $\frac{3}{1-\gamma_i}\Tr(E_i Y_{\tau_i}E_i^\top) \leq 3 \lambda$ in the last inequality. Using this and $\log\frac{\det(W_{\tau_i})}{\det(\lambda I_{n+m})}\leq (n+m) \log T$, we can write
\begin{align*}
\|\Delta_{\tau_i}\|^2\leq \,&\frac{2+k^2}{\tau_i}\Big(240n(n+m)\log(4T^4)+\frac{240\lambda(n+m)\phi^2}{\sigma^2}\Big)\\
\leq \,&\frac{(2+k^2)(n+m)}{\tau_i}\Big(960n\log(4T)+\frac{240\lambda\phi^2}{\sigma^2}\Big)\\
\leq \,&\frac{(2+k^2)(n+m)}{\tau_i}\frac{240\lambda(1+\phi^2) }{\sigma^2}\leq \frac{\epsilon_0^2\tau_1}{\tau_i}\leq \epsilon_0^2 4^{-i+1},
\end{align*}
where we have used the definition of $\tau_1$ in the last inequality.
\end{proof}

\begin{lemma}\label{lem:513}
On the event $\mathcal{E}$, we have that $\max\{1, \log_2\frac{\mu_1}{\mu_*}\}\leq n_s \leq 2+ \max\{1, \log_2\frac{\mu_1}{\mu_*}\}$.
\end{lemma}

\begin{proof}
The proof is given in \cite[Lemma~23]{AC-AC-TK:2020}.
\end{proof}

\begin{lemma}\label{lem:512}
On the event $\mathcal{E}$, we have that
\begin{enumerate}
\item $K_{\tau_i}$ is $(k,\ell)$-strongly stable, for all $1\leq i \leq n_T$;
\item $\|x_t\|^2\leq x_b$, for all $1\leq t \leq T$;
\item $\|u_t\|^2\leq \lambda$ for all $1\leq t \leq T$;
\item $K_{\tau_i}K_{\tau_i}^\top \succeq \frac{\mu_*}{2}I$ for $n_s\leq i \leq n_T$;
\item $\|\Delta_{\tau_i}\| \leq \epsilon_0 2^{-i+1}$, for all $n_s\leq i \leq n_T$;
\item $\widehat{W}_{\tau_i}\succeq \frac{\sigma^2\tau_{i}}{\max\{40(2+k^2),\frac{20}{\mu_*p}(\mu_*+2p+2)\}}I_{n+m}$ for all $1\leq i \leq n_T$,
\end{enumerate}
where 
\begin{equation}\label{x_b2}
x_b=135 n k^2 \sigma^2 \max\Big\{(1+\phi)^2k^6,4k^6\Big\}\log(4T),
\end{equation}
and $\lambda=(1+k)^2 x_b$.

\end{lemma}

\begin{proof}
The proof of Lemma~\ref{lem:512} is similar to the proof of Lemma~\ref{lem210}, and uses an induction. Parts 1, 2, 3, 5, and 6 for $1\leq i\leq\tau_{n_s}$ can be proved using similar arguments of Lemma~\ref{lem210}.  
For part 4, by Lemma~\ref{lem:505}, we have $\|\Delta_{\tau_i}\| \leq \epsilon_0 2^{-n_s+1}$ and by using Lemma~\ref{lem:513}, we have $\|\Delta_{\tau_{n_s}}\|\leq \min\big\{\epsilon_0,\frac{\mu_*}{4kC_0}\big\}$. Now we can use Lemma~\ref{lem42} to prove that $K_{\tau_{n_s}}K_{\tau_{n_s}}^\top \succeq \frac{\mu_*}{2}I$. 
Next we prove the statements of Lemma~\ref{lem:512}, for $i=n_s+1,\ldots, n_T$. Assume the statements of Lemma~\ref{lem:512} are true for $i=k$ for $n_s \leq k< n_T$, we show that these are also true for $i=k+1$. 

Now for $\widehat{W}_{\tau_{i+1}}$ and by similar arguments as in Lemma~\ref{lem210}, we have that
\begingroup
\allowdisplaybreaks
\begin{align*}
\widehat{W}_{\tau_{i+1}}=\,& \begin{pmatrix}
\sum_{s=1}^{\tau_{i+1}-1}x_s x_s^\top + \lambda I_n + \frac{\gamma_{i+1}}{1-\gamma_{i+1}} Y_{\tau_{i+1}}\,& \sum_{s=1}^{\tau_{i+1}-1} x_s u_s^\top\\[10pt]
\sum_{s=1}^{\tau_{i+1}-1}u_s x_s^\top \,& \sum_{s=1}^{\tau_{i+1}-1} u_s u_s^\top + \lambda I_m 
\end{pmatrix}\\[10pt]
\succeq\,&\widehat{W}_{\tau_{i}}+\begin{pmatrix}
I_n \\ K_{\tau_i}
\end{pmatrix}
\sum_{s=\tau_i}^{\tau_{i+1}-1}x_{s}x_{s}^\top \begin{pmatrix}
I_n \,& K_{\tau_i}^\top
\end{pmatrix}+\begin{pmatrix}
\frac{\gamma_{i+1}-\gamma_i}{(1-\gamma_{i+1})(1-\gamma_i)} Y_{\tau_1} \,& 0\\[10pt]
0 \,& 0
\end{pmatrix}
\end{align*}
\endgroup
Using Lemma~\ref{thm20} and the fact that $\mathbb{E}[x_s x_s^\top]\succeq \sigma^2 I_n$, for the event $\mathcal{E}_x$, we have $\sum_{s=\tau_{i}}^{\tau_{i+1}-1}x_{s}x_{s}^\top\succeq \frac{(\tau_{i+1}-\tau_i)\sigma^2}{40}I_n$, and hence we have that
\begin{align*}
\widehat{W}_{\tau_{i+1}}\succeq \,& \widehat{W}_{\tau_{i}}+\frac{(\tau_{i+1}-\tau_{i})\sigma^2}{40}\begin{pmatrix}
I_n \\ K_{\tau_i}
\end{pmatrix}\begin{pmatrix}
I_n \,& K_{\tau_i}^\top
\end{pmatrix}+\begin{pmatrix}
\frac{\gamma_{i+1}-\gamma_i}{(1-\gamma_{i+1})(1-\gamma_i)} \frac{\tau_1 \sigma^2}{40(2+k^2)}I \,& 0\\[10pt]
0 \,& 0
\end{pmatrix}\\[10pt]
= \,& \widehat{W}_{\tau_{i}}+ \frac{\sigma^2(\tau_{i+1}-\tau_i)}{40}\begin{pmatrix}
I_n + \frac{\gamma_{i+1}-\gamma_i}{(1-\gamma_{i+1})(1-\gamma_i)}\frac{\tau_1}{(\tau_{i+1}-\tau_i)}\frac{1}{2+k^2}I_n \,& K_{\tau_i}^\top\\[10pt]
K_{\tau_i} \,& K_{\tau_i}K_{\tau_i}^\top 
\end{pmatrix}.
\end{align*}
Now given that $\tau_i=4^{i-1}\tau_1$, we have $\frac{\tau_1}{(\tau_{i+1}-\tau_i)}=\frac{1}{3(4^{i-1})}$. Let $\gamma_i$ be such that $(1-\gamma_{i})\leq \frac{1}{4^{i}}$ and $1-\gamma_{i+1}\leq \frac{1}{4}(1-\gamma_i)$. Then we have 
\begin{align*}
\frac{\gamma_{i+1}-\gamma_i}{(1-\gamma_{i+1})(1-\gamma_i)}=\,&\frac{1}{1-\gamma_{i+1}}-\frac{1}{1-\gamma_i}\\
\geq\,& \frac{3}{1-\gamma_i}\geq 3(4^{i}).
\end{align*}
By defining $p=\frac{4}{(2+k^2)}$, we have that \[
p\leq \frac{\gamma_{i+1}-\gamma_i}{(1-\gamma_{i+1})(1-\gamma_i)}\frac{\tau_1}{(\tau_{i+1}-\tau_i)}\frac{1}{2+k^2},\]
and also
\begin{align*}
\widehat{W}_{\tau_{i+1}}\succeq \widehat{W}_{\tau_{i}}+ \frac{\sigma^2(\tau_{i+1}-\tau_i)}{40}\begin{pmatrix}
I_n + pI_m \,& K_{\tau_i}^\top\\[10pt]
K_{\tau_i} \,& K_{\tau_i}K_{\tau_i}^\top 
\end{pmatrix}.
\end{align*}
By Lemma~\ref{lem2} and using the induction hypothesis $\widehat{W}_{\tau_{i}}\succeq \frac{\sigma^2\tau_i}{\max\{40(2+k^2),\frac{20}{\mu_*p}(\mu_*+2p+2)\}}I_{n+m}$ and $K_{\tau_i}K_{\tau_i}^\top \succeq \frac{\mu_*}{2} I$, we have that
\begin{align*}
\widehat{W}_{\tau_{i+1}}\succeq \,&\frac{\sigma^2\tau_i}{\max\{40(2+k^2),\frac{20}{\mu_*p}(\mu_*+2p+2)\}}I_{n+m}+ \frac{\sigma^2(\tau_{i+1}-\tau_i)\mu_*p}{20(\mu_*+2p+2)}I_{n+m}\\
\succeq \,&\frac{\sigma^2\tau_{i+1}}{\max\{40(2+k^2),\frac{20}{\mu_*p}(\mu_*+2p+2)\}}I_{n+m}.
\end{align*}

Next we can show that $\|x_t\|^2\leq x_b$, $\|u_t\|^2\leq k^2 x_b$, and $\|z_t\|^2\leq \lambda$ for all $\tau_i\leq t < \tau_{i+1}$, and $\|\Delta_{\tau_{i+1}}\|\leq \epsilon_0 r^{-i+1}$, where $x_b$ is given by~\eqref{x_b2} and $\lambda=(1+k)^2 x_b$ by a similar argument as in Lemma~\ref{lem210}.
Finally, by applying Lemma~\ref{lem42} in Appendix A., we conclude that $K_{\tau_{i+1}}$ is $(k,\ell)$-strongly stable and $K_{\tau_{i+1}}K_{\tau_{i+1}}^\top\succeq \frac{\mu_*}{2} I$.
\end{proof}

\begin{proof}[of Theorem~\ref{thm2}]
Let the event $\mathcal{E}=\mathcal{E}_x\cap \mathcal{E}_W \cap \mathcal{E}_w \cap \mathcal{E}_\eta \cap \mathcal{E}_\Delta$. The regret can be written as 
\begin{align*}
\mathbb{E}[\mathcal{R}_T]={J}_1+ {J}_2 + {J}_3-TJ_*,
\end{align*}
where 
\begin{align*}
J_1=\,&\mathbb{E}\Big[\textbf{1}\{\mathcal{E}\}\sum_{i=1}^{n_T}\sum_{t=\tau_i}^{\tau_{i+1}-1}x_t^\top Q x_t+u_t^\top R u_t\Big], \\
J_2=\,&\mathbb{E}\Big[\textbf{1}\{\mathcal{E}^c\}\sum_{t=\tau_{n_s}}^{T}x_t^\top Q x_t+u_t^\top R u_t\Big], \\
J_3=\,&\mathbb{E}\Big[\sum_{t=0}^{\tau_{n_s}-1}x_t^\top Q x_t+u_t^\top R u_t\Big].
\end{align*}
In the following lemmas, we will bound each term.
\begin{lemma}\label{lemj12}
$J_1\leq TJ_*+n_T(3C_0 \epsilon_0^2 \tau_1+4 \alpha_1 k^6 x_b)$,
where $C_0$ and $\epsilon_0$ are positive constants given in Lemma~\ref{lemm}, and $x_b$ is given in~\eqref{x_b2}.
\end{lemma}
\begin{proof}
The proof is identical to Lemma~\ref{lemj1}.
\end{proof}

\begin{lemma}\label{lemj22}
$J_2\leq (J(K_0)+2\alpha_1 k^2 x_b)T^{-1}+8\alpha_1 k^8(1+8\phi^2)x_b T^{-2}$.
\end{lemma}
\begin{proof}
The proof is identical to Lemma~\ref{lemj2}.
\end{proof}

\begin{lemma}\label{lemj32}
$J_3\leq (1+\phi^2) \left(65J(K_0)\max\big\{1,\frac{\mu^2_1}{\mu_*^2}\big\}\tau_1+8\alpha_1 m \sigma^2 k^{14} \log^2 3T\right)$
\end{lemma}
\begin{proof}
It can be proved using the same proof as in Lemma~29 of \cite{AC-AC-TK:2020}.
\end{proof}

By applying Lemmas~\ref{lemj1}, \ref{lemj2}, and \ref{lemj3}, we conclude that
\begin{align}
\mathbb{E}[\mathcal{R}_T]=&n_T(3C_0 \epsilon_0^2 \tau_1+4 \alpha_1 k^6 x_b)\nonumber \\
	&+(J(K_0)+2\alpha_1 k^2 x_b)T^{-1}+8\alpha_1 k^8(1+8\phi^2)x_b T^{-2}\nonumber\\
	&+(1+\phi^2) \left(65J(K_0)\max\big\{1,\frac{\mu^2_1}{\mu_*^2}\big\}\tau_1+8\alpha_1 m \sigma^2 k^{14} \log^2 3T\right)\label{eq:911}
\end{align}
By substituting the values of $\tau_1$, $x_b$, and $n_T$ in~\eqref{eq:911}, we have that
\begingroup
\allowdisplaybreaks
\begin{align}
\mathbb{E}[\mathcal{R}_T]\leq\,& 135\log(T)\Big(3C_0 (240(1+k^2)(1+\phi^2)\max\{(2+k^2),\frac{1}{2\mu_*p}(\mu_*+2p+2)\}(n+m)+\sigma^2)\nonumber\\
\,&+4 \alpha_1 k^6\sigma^2\Big)nk^2\max\{(1+\phi)^2k^6,4k^6\}\log(4T)\nonumber\\
\,&+(J(K_0)+270\alpha_1 nk^4\sigma^2\max\{(1+\phi)^2k^6,4k^6\}\log(4T))T^{-1}\nonumber\\
\,&+1080\alpha_1 (1+8\phi^2)nk^{10}\sigma^2\max\{(1+\phi)^2k^6,4k^6\}\log(4T) T^{-2}\nonumber\\
\,&+135\frac{15600(1+k^2)(1+\phi^2)\max\{(2+k^2),\frac{1}{2\mu_*p}(\mu_*+2p+2)\}(n+m)+\sigma^2}{\epsilon_0^2}\nonumber\\
\,&nk^2\max\{(1+\phi)^2k^6,4k^6\}\log(4T)(1+\phi^2)\nu \max\big\{1,\frac{\mu^2_1}{\mu_*^2}\big\}+8(1+\phi^2)\alpha_1 m \sigma^2 k^{14} \log^2 3T.\nonumber
\end{align}
\endgroup
By rearranging the terms we get 
\begingroup
\allowdisplaybreaks
\begin{align}
\mathbb{E}[\mathcal{R}_T]\leq\,& 135\Big(3C_0 (240(1+k^2)(1+\phi^2)\max\{(2+k^2),\frac{1}{2\mu_*p}(\mu_*+2p+2)\}(n+m)+\sigma^2)\nonumber\\
\,&+4 \alpha_1 k^6\sigma^2+\alpha_1 m \sigma^2 k^{8}\Big)nk^2\max\{(1+\phi)^2k^6,4k^6\}\log^2(T)\nonumber\\
\,&+135\frac{240(1+k^2)(1+\phi^2)\max\{(2+k^2),\frac{1}{2\mu_*p}(\mu_*+2p+2)\}(n+m)+\sigma^2}{\epsilon_0^2}\nonumber\\
\,&~~nk^2\max\{(1+\phi)^2k^6,4k^6\}(1+\phi^2)\nu\log(4T) \nonumber\\
\,&+{135}\Big(3C_0 (240(1+k^2)(1+\phi^2)\max\{(2+k^2),\frac{1}{2\mu_*p}(\mu_*+2p+2)\}(n+m)+\sigma^2)\nonumber\\
\,&+4 \alpha_1 k^6\sigma^2\Big)nk^2\max\{(1+\phi)^2k^6,4k^6\}\log(4)\log(T)\nonumber\\
\,&+(\nu+270\alpha_1 nk^4\sigma^2\max\{(1+\phi)^2k^6,4k^6\}\log(4T))T^{-1}\nonumber\\
\,&+1080\alpha_1 (1+8\phi^2)nk^{10}\sigma^2\max\{(1+\phi)^2k^6,4k^6\}\log(4T) T^{-2}.\label{eq:reg1}
\end{align}
\endgroup
Overall, noting that the dominant term is $\log^2(T)$ in the first two lines of~\eqref{eq:reg1}, we conclude that \[\mathbb{E}[\mathcal{R}_T]\leq \mathrm{poly} (\alpha_0, \alpha_1, \phi, \nu, m, n , r)\log^2(T).\]
\end{proof}



\vskip 0.2in
\bibliographystyle{plain}


\begin{thebibliography}{10}

\bibitem{YA-CS:11}
Y.~Abbasi-Yadkori and C.~Szepesv{\'a}ri.
\newblock Regret bounds for the adaptive control of linear quadratic systems.
\newblock In {\em Proceedings of the 24th Annual Conference on Learning
  Theory}, pages 1--26, 2011.

\bibitem{MA-AL:17}
M.~Abeille and A.~Lazaric.
\newblock {Thompson} sampling for linear-quadratic control problems.
\newblock In {\em Artificial Intelligence and Statistics}, pages 1246--1254.
  Proceedings of Machine Learning Research, 2017.

\bibitem{MA-AL:18}
M.~Abeille and A.~Lazaric.
\newblock Improved regret bounds for {Thompson} sampling in linear quadratic
  control problems.
\newblock In {\em Proceedings of the 35th International Conference on Machine
  Learning}, volume~80 of {\em Proceedings of Machine Learning Research}, pages
  1--9, 10--15 Jul 2018.

\bibitem{NA-BB-EH-SK-KS:19}
N.~Agarwal, B.~Bullins, E.~Hazan, S.~Kakade, and K.~Singh.
\newblock Online control with adversarial disturbances.
\newblock In {\em Proceedings of the 36th International Conference on Machine
  Learning}, volume~97, pages 111--119, 2019.

\bibitem{NA-EH-KS:19}
N.~Agarwal, E.~Hazan, and K.~Singh.
\newblock Logarithmic regret for online control.
\newblock {\em arXiv preprint arXiv:1909.05062}, 2019.

\bibitem{MA-BG-TL:20-pmlr}
M.~Akbari, B.~Gharesifard, and T.~Linder.
\newblock {Riccati} updates for online linear quadratic control.
\newblock In {\em Proceedings of the 2nd Conference on Learning for Dynamics
  and Control}, volume 120, pages 476--485. Proceedings of Machine Learning
  Research, 2020.

\bibitem{DB:2005}
D.~Bertsekas.
\newblock Dynamic programming and optimal control: Volume {I}.
\newblock 2012.

\bibitem{DPB-SES:96}
D.~P. Bertsekas and S.~E. Shreve.
\newblock {\em Stochastic optimal control: the discrete-time case}, volume~5.
\newblock Athena Scientific, 1996.

\bibitem{SB-MCC:06}
S.~Bittanti and M.~C. Campi.
\newblock Adaptive control of linear time invariant systems: the “bet on the
  best” principle.
\newblock {\em Communications in Information \& Systems}, 6(4):299--320, 2006.

\bibitem{MCC-PRK:98}
M.~C. Campi and P.~R. Kumar.
\newblock Adaptive linear quadratic {Gaussian} control: the cost-biased
  approach revisited.
\newblock {\em SIAM Journal on Control and Optimization}, 36(6):1890--1907,
  1998.

\bibitem{AC-AC-TK:2020}
A.~Cassel, A.~Cohen, and T.~Koren.
\newblock Logarithmic regret for learning linear quadratic regulators
  efficiently.
\newblock In {\em International Conference on Machine Learning}, pages
  1328--1337. Proceedings of Machine Learning Research, 2020.

\bibitem{HFC-LG:87}
H.~Chen and L.~Guo.
\newblock Optimal adaptive control and consistent parameter estimates for
  {ARMAX} model with quadratic cost.
\newblock {\em SIAM Journal on Control and Optimization}, 25(4):845--867, 1987.

\bibitem{HFC-JFZ:90}
H.~Chen and J.~Zhang.
\newblock Identification and adaptive control for systems with unknown orders,
  delay, and coefficients.
\newblock {\em IEEE Transactions on Automatic Control}, 35(8):866--877, 1990.

\bibitem{AC-AH-TK-NL-YM-KT:18}
A.~Cohen, A.~Hasidim, T.~Koren, N.~Lazic, Y.~Mansour, and K.~Talwar.
\newblock Online linear quadratic control.
\newblock In {\em Proceedings of the 35th International Conference on Machine
  Learning}, volume~80, pages 1029--1038, 2018.

\bibitem{AC-TK-YM:19}
A.~Cohen, T.~Koren, and Y.~Mansour.
\newblock Learning linear-quadratic regulators efficiently with only
  {$\sqrt{T}$} regret.
\newblock In {\em International Conference on Machine Learning}, pages
  1300--1309. Proceedings of Machine Learning Research, 2019.

\bibitem{RC:87}
R.~Cristi.
\newblock Internal persistency of excitation in indirect adaptive control.
\newblock {\em IEEE Transactions on Automatic Control}, 32(12):1101--1103,
  1987.

\bibitem{OD-AF-NH-PJ:17}
O.~Dekel, A.~Flajolet, N.~Haghtalab, and P.~Jaillet.
\newblock Online learning with a hint.
\newblock In {\em Advances in Neural Information Processing Systems},
  volume~30, 2017.

\bibitem{VVD-EAL:03}
V.~V. Dombrovskii and E.~A. Lyashenko.
\newblock A linear quadratic control for discrete systems with random
  parameters and multiplicative noise and its application to investment
  portfolio optimization.
\newblock {\em Automation and Remote Control}, 64(10):1558--1570, 2003.

\bibitem{MG-JBM:86}
M.~Green and J.~B. Moore.
\newblock Persistence of excitation in linear systems.
\newblock {\em Systems \& Control Letters}, 7(5):351--360, 1986.

\bibitem{ME-AJ-BR:12}
M.~Ibrahimi, A.~Javanmard, and B.~Roy.
\newblock Efficient reinforcement learning for high dimensional linear
  quadratic systems.
\newblock {\em Advances in Neural Information Processing Systems}, 25, 2012.

\bibitem{TK-SL-KA-AA-BH:22}
T.~Kargin, S.~Lale, K.~Azizzadenesheli, A.~Anandkumar, and B.~Hassibi.
\newblock {Thompson} sampling achieves {$\tilde O(\sqrt{T})$} regret in linear
  quadratic control.
\newblock In {\em Conference on Learning Theory}, pages 3235--3284. Proceedings
  of Machine Learning Research, 2022.

\bibitem{TLL-CZW:82}
T.~Lai and C.~Wei.
\newblock Least squares estimates in stochastic regression models with
  applications to identification and control of dynamic systems.
\newblock {\em The Annals of Statistics}, 10(1):154--166, 1982.

\bibitem{JML-LFP:12}
J.~M. Lemos and L.~F. Pinto.
\newblock Distributed linear-quadratic control of serially chained systems:
  application to a water delivery canal [applications of control].
\newblock {\em IEEE Control Systems Magazine}, 32(6):26--38, 2012.

\bibitem{TPL-JJH-AP:15}
T.~P. Lillicrap, J.J. Hunt, A.~Pritzel, N.~Heess, T.~Erez, Y.~Tassa, D.~Silver,
  and D.~Wierstra.
\newblock Continuous control with deep reinforcement learning.
\newblock {\em arXiv preprint arXiv:1509.02971}, 2015.

\bibitem{TTL-SHS:02}
T.~Lu and S.~Shiou.
\newblock Inverses of $2\times 2$ block matrices.
\newblock {\em Computers \& Mathematics with Applications}, 43(1-2):119--129,
  2002.

\bibitem{MM-MSK-AH-SV:18}
M.~Mahmud, M.~S. Kaiser, A.~Hussain, and S.~Vassanelli.
\newblock Applications of deep learning and reinforcement learning to
  biological data.
\newblock {\em IEEE Transactions on Neural Networks and Learning Systems},
  29(6):2063--2079, 2018.

\bibitem{HM-ST-BR:19}
H.~Mania, S.~Tu, and B.~Recht.
\newblock Certainty equivalence is efficient for linear quadratic control.
\newblock In {\em Advances in Neural Information Processing Systems},
  volume~32, 2019.

\bibitem{ZKO-PB-JS:07}
K.~Z. Ostergaard, P.~Brath, and J.~Stoustrup.
\newblock Gain-scheduled linear quadratic control of wind turbines operating at
  high wind speed.
\newblock In {\em 2007 IEEE International Conference on Control Applications},
  pages 276--281. IEEE, 2007.

\bibitem{YO-MG-RJ:17}
Y.~Ouyang, M.~Gagrani, and R.~Jain.
\newblock Control of unknown linear systems with {Thompson} sampling.
\newblock In {\em 2017 55th Annual Allerton Conference on Communication,
  Control, and Computing (Allerton)}, pages 1198--1205. IEEE, 2017.

\bibitem{BR:19}
B.~Recht.
\newblock A tour of reinforcement learning: The view from continuous control.
\newblock {\em Annual Review of Control, Robotics, and Autonomous Systems},
  2:253--279, 2019.

\bibitem{FR-SC-FQ:13}
F.~Rinaldi, S.~Chiesa, and F.~Quagliotti.
\newblock Linear quadratic control for quadrotors {UAVs} dynamics and formation
  flight.
\newblock {\em Journal of Intelligent \& Robotic Systems}, 70(1):203--220,
  2013.

\bibitem{DS-AH:16}
D.~Silver, A.~Huang, C.~J. Maddison, A.~Guez, L.~Sifre, G.~Van Den~Driessche,
  J.~Schrittwieser, I.~Antonoglou, V.~Panneershelvam, M.~Lanctot, et~al.
\newblock Mastering the game of {Go} with deep neural networks and tree search.
\newblock {\em Nature}, 529(7587):484--489, 2016.

\bibitem{MS-DF:20}
M.~Simchowitz and D.~Foster.
\newblock Naive exploration is optimal for online {LQR}.
\newblock In {\em International Conference on Machine Learning}, pages
  8937--8948. Proceedings of Machine Learning Research, 2020.

\bibitem{VVB-LC-GB-VW-VW:96}
V.~Van~Breusegem, L.~Chen, G.~Bastin, V.~Wertz, V.~Werbrouck, and
  C.~de~Pierpont.
\newblock An industrial application of multivariable linear quadratic control
  to a cement mill circuit.
\newblock {\em IEEE Transactions on Industry Applications}, 32(3):670--677,
  1996.

\bibitem{JCW-PR-IM-BLMD:05}
J.C. Willems, P.~Rapisarda, I.~Markovsky, and B.L.M. De~Moor.
\newblock A note on persistency of excitation.
\newblock {\em Systems \& Control Letters}, 54(4):325--329, 2005.

\end{thebibliography}
\newpage

\appendix
\section*{Appendix A.}\label{appdx}
\label{app:theorem}

\begin{lemma}\cite[Lemma~34]{AC-AC-TK:2020}\label{lem13}
Let $w_t\in \real^n$ for $t=1,\ldots, T$ be i.i.d. random variables with distribution $\mathcal{N}(0,\sigma^2 I_n)$. Suppose that $T>2$. Then with probability at least $1-\delta$ we have that
\begin{equation}
\max_{1\leq t \leq T}\|w_t\|\leq \sigma \sqrt{5n\log\frac{T}{\delta}}
\end{equation}
\end{lemma}

\begin{lemma}\cite[Lemma~37]{AC-AC-TK:2020}\label{lem105}
Let $z_s\in\real^m$ for $s=1,\ldots, t-1$ be such that $\|z_s\|^2\leq \lambda$ and define $W_t=\lambda I + \sum_{s=1}^{t-1}z_s z_s^\top$. Then we have that
\begin{equation}
\log\frac{\det(W_t)}{\det(\lambda I)}\leq m \log t.
\end{equation}
\end{lemma}

\begin{lemma}\cite[Lemma~38]{AC-AC-TK:2020}\label{lem104}
Suppose $K$ is $(k,\ell)$-strongly stable controller and $s_0$, $s_1$ are integers such that $1\leq s_0 < s_1\leq T$. Let $x_s$ for $s=s_0, \ldots, s_1$ be the sequence of states generated under the controller $K$ starting from $x_{s_0}$. Then we have 
\begin{equation}
\|x_t\|\leq k(1-l)^{t-s_0}\|x_{s_0}\|+\frac{k}{l}\max_{1\leq t\leq T}\|w_t\|, \qquad \mathrm{for} \ \mathrm{all} \ s_0\leq t \leq s_1.
\end{equation}
\end{lemma}

\begin{lemma}\cite[Lemma~39]{AC-AC-TK:2020}\label{lem39}
Suppose $K_1, \ldots , K_r$ are $(k,\ell)$-strongly stable feedback gains and $\{t_i\}_{i=1}^{r+1}$ are integers such that $1\leq t_1 < \ldots <t_{r+1}\leq T $. For each $t_i$ Let $\{x_t\}$ be the sequence of states generated by starting from $x_{t_i}$ and playing controller $K_i$ at times $t_i\leq t < t_{i+1}$, i.e., $x_{t+1}=(A_*+B_*K_i)x_t+w_t$ for all $t_i\leq t< t_{i+1}$. Denote $\tau=\min_i\{t_{i+1}-t_i\}$ and suppose that $\tau\geq l^{-1}\log(2k)$. Then we have 
\begin{align*}
\|x_t\|\leq 3k\max\Big\{\frac{1}{2}\|x_{t_1}\|, \frac{k}{l}\max_{1\leq s\leq T}\|w_s\|\Big\}, \quad \mathrm{for} \ \mathrm{all} \ t_1\leq t\leq t_{r+1}
\end{align*}
\end{lemma}

\begin{lemma}\cite[Lemma~40]{AC-AC-TK:2020}\label{lem40}
Suppose $K$ is a $(k,\ell)$-strongly stable controller and let $x_s$ for $s=1,\ldots, t$ be the sequence of states generated under the control $K$ starting from $x_1$, i.e., $x_{s+1}=(A_*+B_*K)x_s+ w_s$ for all $1\leq s\le t$. Then we have that
$\mathbb{E}\Big[\sum_{s=1}^t x_s^\top(Q+K^\top R K)x_s\Big|x_1\Big]\leq t J(K)+\frac{2\alpha_1 k^4}{l}\|x_1\|^2$.
\end{lemma}

\begin{lemma}\cite[Lemma~41]{AC-AC-TK:2020}\label{Lem41}
Suppose $J(K)<J$, then $K$ is $(k,\ell)$-strongly stable with $k=\frac{J}{\alpha_0\sigma^2}$ and $\ell=\frac{\alpha_0 \sigma^2}{2J}$
\end{lemma}

\begin{lemma}\cite[Lemma~42]{AC-AC-TK:2020}\label{lem42}
Let $A\in \real^{n\times n}$, $B\in \real^{n \times m}$ and denote $\Delta=\|(A-A_\star \quad B-B_\star)\|$. Taking $K=\mathrm{dare}(A,B,Q,R)$ and denoting $k=\sqrt{\frac{\nu+C_0\epsilon_0^2}{\alpha_0\sigma^2}}$, and $l=\frac{1}{2k^2}$ where $J_\star\leq\nu$, and $Q,R\succeq \alpha_0 I$, we have that
\begin{enumerate}
\item If $\Delta\leq \epsilon_0$ then $K$ is $(k,\ell)$-strongly stable;
\item If $\Delta\leq \min\big\{\epsilon_0,\frac{\mu}{4kC_0}\big\}$ then $KK^\top\succeq K_*K_*^\top-\frac{\mu}{2}I$ and $K_*K_*^\top \succeq KK^\top-\frac{\mu}{2}I$;
\item If $\Delta\leq \min\big\{\epsilon_0,\frac{\mu_*}{4kC_0}\big\}$ then $KK^\top \succeq \frac{\mu_*}{2}I$.
\end{enumerate} 
\end{lemma}

\begin{lemma}\label{lem43}
Let $X_1$ and $X_2$ be two positive definite block matrices as follows.
\begin{align*}
X_1=\begin{pmatrix}
A_1 \,& B_1 \\
B_1^\top \,& C_1
\end{pmatrix} \qquad \qquad X_2=\begin{pmatrix}
A_2 \,& B_2 \\
B_2^\top \,& C_2
\end{pmatrix},
\end{align*}
If $X_1-X_2$ is positive semi-definite, then we have that 
\[
\big(C_1-B_1^\top A_1^{-1}B_1\big)-\big(C_2-B_2^\top A_2^{-1}B_2\big)\succeq 0\]
\end{lemma}
\begin{proof}
Note that for a positive definite block matrix we have that
\begin{align*}
\begin{pmatrix}
x^\top \,& y^\top
\end{pmatrix} \begin{pmatrix}
A \,& B\\
B^\top \,& C
\end{pmatrix}\begin{pmatrix}
x \\ y
\end{pmatrix}=(x+A^{-1}By)^\top A (x+A^{-1}By)+y^\top(C-B^\top A^{-1} B)y,
\end{align*}
for all real vectors $(x \ y)$. Given this and by assumption $X_1-X_2\succeq 0$, we have that
\begin{align*}
(x+A_1^{-1}B_1y)^\top A_1 (x+A_1^{-1}B_1y)+y^\top(C_1-B_1^\top A_1^{-1} B_1)y\geq \\
(x+A_2^{-1}B_2y)^\top A_2 (x+A_2^{-1}B_2y)+y^\top(C_2-B_2^\top A_2^{-1} B_2)y,
\end{align*}
Now for any $y$ and choosing $x=-A_1^{-1}B_1 y$, we have that
\begin{align*}
y^\top(C_1-B_1^\top A_1^{-1} B_1)y\geq \,&(-A_1^{-1}B_1 y+A_2^{-1}B_2y)^\top A_2 (-A_1^{-1}B_1 y+A_2^{-1}B_2y)\\\,&+y^\top(C_2-B_2^\top A_2^{-1} B_2)y\\
\geq \,& y^\top(C_2-B_2^\top A_2^{-1} B_2)y
\end{align*}
where we have used the fact that $A_2$ is positive definite. This implies the result.

\end{proof}

\section*{Appendix B.}\label{appdxC}
Here, we consider an alternative algorithm. In this algorithm, $\widehat{B}_{\tau_i}$ is not updated using Line 7 of Algorithm~\ref{alg}. Instead, at time $t=\tau_i$, an estimate $\widehat{B}_{\tau_i}$ is given to the controller which is in an $\epsilon$ neighbourhood of $B_*$, where $0\leq \epsilon\leq \frac{\epsilon_0}{4k}$. The controller receives the hint on the direction towards $B_*$ (Line 8 of Algorithm~\ref{alg}), updates $B_{\tau_i}$ by
\[
B_{\tau_i}=\widehat{B}_{\tau_i}-\gamma_i(\widehat{B}_{\tau_i}-B_*)+E_i,
\]
and the estimate $A_{\tau_i}$ using Line 9 of Algorithm~\ref{alg}. We assume that the system and control action are scalar, i.e., $n=m=1$. 
Hence, by assuming $1-\gamma_i\leq r^{-2i}$, $|E_i|^2\leq \frac{r^{-2i}}{4k\tau_i}$, and $|\widehat{B}_{\tau_i}-B_*|\leq \epsilon$, we have that
\begin{align}
|B_{\tau_i}-B_*|\leq &(1-\gamma_i)|\big(\widehat{B}_{\tau_i}-B_*\big)|+|E_i|\nonumber\\
\leq & \epsilon r^{-2i} + \frac{r^{-2i}}{4k\tau_i}\leq \frac{\epsilon_0}{2k} r^{-2i}\label{eqq92},
\end{align} 
and using Line 9 of Algorithm~\ref{alg}, we have
\begin{align*}
A_{\tau_i}=&\Big(\sum_{s=1}^{\tau_i-1}(x_{s+1}-B_{\tau_i}u_s)x_s\Big)\Big(\sum_{s=1}^{\tau_i-1}x_s^2+\lambda I_n\Big)^{-1}\\
=&\Big(\sum_{s=1}^{\tau_i-1}(A_*x_s+B_*u_s+w_s-B_{\tau_i}u_s)x_s\Big)V_{\tau_i}^{-1},
\end{align*}
where we have used~\eqref{eq102} and $V_{\tau_i}=\sum_{s=1}^{\tau_i-1}x_s^2+\lambda I_n$. Then we have
\begin{align*}
A_{\tau_i}=&(\sum_{s=1}^{\tau_i-1}A_*x_s^2)V_{\tau_i}^{-1}+(\sum_{s=1}^{\tau_i-1}(B_*-B_{\tau_i})u_s x_s)V_{\tau_i}^{-1} +(\sum_{s=1}^{\tau_i-1}w_s x_s)V_{\tau_i}^{-1}\\
=& A_*(\sum_{s=1}^{\tau_i-1}x_s^2)V_{\tau_i}^{-1}+(B_*-B_{\tau_i})(\sum_{s=1}^{\tau_i-1}u_s x_s)V_{\tau_i}^{-1}+(\sum_{s=1}^{\tau_i-1}w_s x_s)V_{\tau_i}^{-1}\\
=& A_*-\lambda A_* V_{\tau_i}^{-1}+(B_*-B_{\tau_i})(\sum_{s=1}^{\tau_i-1}u_s x_s)V_{\tau_i}^{-1}+(\sum_{s=1}^{\tau_i-1}w_s x_s)V_{\tau_i}^{-1}.
\end{align*}
Now using the controller policy $u_{s}=K_{0}x_{s}+\eta_{s}$ for $1\leq s\leq \tau_{1}-1$, and $u_{s}=K_{\tau_i}x_{s}$ for $\tau_i\leq s\leq \tau_{i+1}-1$, we have
\begin{align}
A_{\tau_i}=& A_*-\lambda A_* V_{\tau_i}^{-1}+(B_*-B_{\tau_i})(\sum_{r=1}^i\sum_{s=\tau_{r-1}}^{\tau_r-1}(K_{\tau_{r-1}}x_{s}+\eta_{s}) x_s)V_{\tau_i}^{-1}+(\sum_{s=1}^{\tau_i-1}w_s x_s)V_{\tau_i}^{-1}\nonumber\\
=& A_*-\lambda A_* V_{\tau_i}^{-1}+(B_*-B_{\tau_i})(\sum_{r=1}^i\sum_{s=\tau_{r-1}}^{\tau_r-1}K_{\tau_{r-1}}x_{s}x_s)V_{\tau_i}^{-1} +(B_*-B_{\tau_i})(\sum_{s=1}^{\tau_1-1}\eta_{s}x_s)V_{\tau_i}^{-1}\nonumber\\
&+(\sum_{s=1}^{\tau_i-1}w_s x_s)V_{\tau_i}^{-1}\nonumber\\
=& A_*-\lambda A_* V_{\tau_i}^{-1}+(B_*-B_{\tau_i})(\sum_{r=1}^i\sum_{s=\tau_{r-1}}^{\tau_r-1}K_{\tau_{r-1}}x_{s}x_s)V_{\tau_i}^{-1}\nonumber\\
&+\left(\sum_{s=1}^{\tau_1-1}(B_*-B_{\tau_i})\eta_{s}x_s+\sum_{s=1}^{\tau_i-1}w_s x_s\right)V_{\tau_i}^{-1}\label{eqq91}.
\end{align}

Defining $\Delta_{A_{\tau_i}}=A_{\tau_i}- A_*$ and $\Delta_{B_{\tau_i}}=B_{\tau_i}- B_*$, we have
\begin{align*}
\Delta_{A_{\tau_i}}=&-\lambda A_* V_{\tau_i}^{-1}+\Delta_{B_{\tau_i}}(\sum_{r=1}^i\sum_{s=\tau_{r-1}}^{\tau_r-1}K_{\tau_{r-1}}x_{s}x_s)V_{\tau_i}^{-1}+\left(\sum_{s=1}^{\tau_1-1}\Delta_{B_{\tau_i}}\eta_{s}x_s+\sum_{s=1}^{\tau_i-1}w_s x_s\right)V_{\tau_i}^{-1}.
\end{align*}

Using $K_{\tau_i}\leq k$, we have
\begin{align*}
|\Delta_{A_{\tau_i}}|&\leq |\lambda A_* V_{\tau_i}^{-1}|+k|\Delta_{B_{\tau_i}}|(\sum_{s=1}^{\tau_i-1}x_{s}x_s)V_{\tau_i}^{-1}+\left|\left(\sum_{s=1}^{\tau_1-1}\Delta_{B_{\tau_i}}\eta_{s}x_s+\sum_{s=1}^{\tau_i-1}w_s x_s\right)V_{\tau_i}^{-1}\right|\\
&\leq |\lambda A_* V_{\tau_i}^{-1}|+k|\Delta_{B_{\tau_i}}| + \lambda k |\Delta_{B_{\tau_i}}| V_{\tau_i}^{-1}+\left|\left(\sum_{s=1}^{\tau_1-1}\Delta_{B_{\tau_1}}\eta_{s}x_s+\sum_{s=1}^{\tau_i-1}w_s x_s\right)V_{\tau_i}^{-1}\right|\\
&\leq |\lambda A_* V_{\tau_i}^{-1}|+\big(\sum_{s=1}^{\tau_i-1}w_s x_s \big)V_{\tau_i}^{-1}+ |\Delta_{B_{\tau_i}}|(k+k\lambda V_{\tau_i}^{-1}) +|\Delta_{B_{\tau_1}}|(\sum_{s=1}^{\tau_1-1}\eta_{s}x_s)V_{\tau_i}^{-1}.
\end{align*}

Similar to Lemma~19 of \cite{AC-AC-TK:2020}, we can bound the first two term of the right-hand side and by using \eqref{eqq92}, we have that

\begin{align*}
|\Delta_{A_{\tau_i}}|^2&\leq \frac{2}{\tau_i}\left(640 n^2\log(3T)+80\frac{\lambda n \phi^2}{\sigma^2} \right)+ \frac{\epsilon_0^2}{2} r^{-2i}(1+\frac{80\lambda}{\sigma^2 \tau_i})\leq \epsilon_0^2 r^{-2i}.
\end{align*}
Using this, Lemma~\ref{lemm}, and following the same proof of \cite[Theorem~1]{AC-AC-TK:2020}, a logarithmic regret will be achieved.



%
%
%

\end{document}